\documentclass[a4paper]{article}

\usepackage{graf_GR}

\newcommand{\alt}{{\widetilde{\alpha}}}
\renewcommand{\Re}{{\mathrm{Re}}}
\renewcommand{\Im}{{\mathrm{Im}}}
\newcommand{\DM}{{D}}
\newcommand{\varept}{{\widetilde{\varep}}}
\renewcommand{\wp}{Q}
\renewcommand{\omo}{\tilde{\omega}}

\renewcommand{\lat}{{\mathchar'26\mkern-10mu \lambda}}

\title{Stationary gravitational modes on Kerr-anti-de Sitter spacetimes}
\author{Olivier Graf\footnote{Univ. Grenoble Alpes, CNRS, IF, 38000 Grenoble, France. E-mail: \texttt{olivier.graf@univ-grenoble-alpes.fr}.}}
\begin{document}
\maketitle

\begin{abstract}
  We prove that, for any continuous path of Kerr-adS black hole parameters crossing the Hawking-Reall threshold, for all azimuthal number $|m|$ sufficiently large, there exists black hole parameters on the path such that the Teukolsky equations with conformal boundary conditions admit a non-trivial regular stationary (\emph{i.e.} with frequency $\omega=m\omega_+$) mode solution. When $|m|$ goes to infinity, we show that these black hole parameters accumulate at the Hawking-Reall threshold. The major difficulty is that one cannot construct solutions to the system of Teukolsky equations using variational arguments as it is usually done for the classical wave equation. We introduce a new scheme of proof based on a shooting-type continuity argument and a high-frequency WKB approximation.
\end{abstract}

\section{Introduction}\label{sec:intro}
The \emph{Kerr-anti-de Sitter spacetimes} (Kerr-adS) are the solutions to the Einstein equations
\begin{align}\label{eq:EEcosmo}
  \RRRic(g) -\half \mathrm{R}(g)g + \Lambda g = 0,
\end{align}
with negative cosmological constant $\La=:-3k^2<0$, given by
\begin{align}\label{eq:gKadS}
  \begin{aligned}
    g_{\mathrm{KadS}} & := -\frac{\Delta}{\Xi^2\Si}\le(\d t - a \sin^2\varth\d\varphi\ri)^2 + \frac{\Si}{\Delta}\d r^2 + \frac{\Si}{\Delta_\varth}\d\varth^2 + \frac{\Delta_\varth}{\Xi^2\Si}\sin^2\varth\le(a\d t - (r^2+a^2)\d\varphi\ri)^2,
  \end{aligned}
\end{align}
where
\begin{align*}
  \Si & := r^2+a^2\cos^2\varth, & \Delta & := (r^2+a^2)\le(1+k^2r^2\ri) - 2Mr, \\
  \Delta_\varth & := 1 - a^2k^2\cos^2\varth, &  \Xi & := 1-a^2k^2.
\end{align*}
We consider \emph{admissible (subextremal) Kerr-adS parameters} $(M,a,k)$, \emph{i.e.} such that
\begin{align*}
  M,k & > 0, & a & \geq 0, & ak & < 1,
\end{align*}
and such that the polynomial $\De(r)$ has two distinct non-negative roots. For such parameters, the expression~\eqref{eq:gKadS} defines a smooth metric on the manifold $\MM_{\mathrm{KadS}} :=  \RRR_t\times(r_+,+\infty)_r\times\SSS^2_{\varth,\varphi}$, where $r=r_+$ is the largest root of $\De(r)$ (which corresponds to the black hole radius).\\ 

The linearisation of the Einstein equations~\eqref{eq:EEcosmo} around a Kerr-adS solution is governed by two wave-type \emph{Teukolsky equations}\footnote{We re-derived~\eqref{eq:Teuk} from~\cite[(3.18)]{Kha83} for the purpose of this introduction. In the bulk of the paper we will use directly the separated versions~\cite[(3.19)--(3.22)]{Kha83}.} \cite{Teu72,Kha83}
\begin{align}\label{eq:Teuk}
  \begin{aligned}
      0 & = \Box_{g_{\mathrm{KadS}}} \al^{[\pm2]} \pm 2\frac{\pr_r\De}{\Si}\pr_r\al^{[\pm2]} \pm 2\Xi\frac{\pr_r\De}{\Si\De}\le(a\pr_\varphi+(r^2+a^2)\pr_t\ri)\al^{[\pm2]} \\
        & \quad \mp 8\frac{\Xi r}{\Si}\pr_t\al^{[\pm2]} + \frac{\le(6k^2r^2 + (1\pm1) \pr^2_r\De\ri)}{\Si}\al^{[\pm2]} \\
        &  \quad \pm 4i\frac{\Xi\cot\varth}{\Si\De_\varth}(1-a^2k^2\cos(2\varth))\le(a\sin\varth\pr_t+\frac{\pr_\varphi}{\sin\varth}\ri)\al^{[\pm2]} \\
        & \quad \mp 8i\frac{\Xi}{\Si} a \cos\varth\pr_t\al^{[\pm2]} -\frac1\Si\le(2\De_\varth+2a^2k^2+4\frac{\Xi^2\cot^2\varth}{\De_\varth}\ri)\al^{[\pm2]},
    \end{aligned}
\end{align}
where
\begin{align*}
  \Box_{g_{\mathrm{KadS}}} \al & = \frac{1}{\Si}\pr_r\le(\De\pr_r\al\ri) - \frac{\Xi^2}{\Si\De}\le(a\pr_\varphi+(r^2+a^2)\pr_t\ri)^2\al \\
                               & \quad + \frac{1}{\Si\sin\varth}\pr_\varth\le(\De_\varth\sin\varth\pr_\varth\al\ri) + \frac{\Xi^2}{\Si\De_\varth} \le(a\sin\varth\pr_t+\frac{\pr_\varphi}{\sin\varth}\ri)^2\al, 
\end{align*}
and where $\al^{[+2]},\al^{[-2]}$ are spin-$\pm2$-weighted complex-valued functions. See~\cite[Section 2.2]{Daf.Hol.Rod19a} or~\cite{Mil24} for a review on spin-weighted functions (which, at first reading, can be thought of as usual complex-valued functions). In this paper, we will renormalise $\al^{[\pm2]}$ as follows
\begin{align}\label{eq:renormalt}
  \alt^{[+2]} & := \De(r^2+a^2)^{1/2} \al^{[+2]}, & \alt^{[-2]} & := \De^{-1}(r^2+a^2)^{1/2}\al^{[-2]}.
\end{align}

Defining the tortoise coordinate $r^\star$ by
\begin{align*}
  \frac{\d r^\star}{\d r} = \frac{r^2+a^2}{\De}, && \text{and} && r^\star(r) \xrightarrow{r\to+\infty} 0,
\end{align*}
we say that $\alt^{[\pm2]}$ are \emph{regular up to infinity} ($r^\star=0$), if there exists an extension of $\alt^{[\pm2]}$ at $r^\star=0$ such that
\begin{align}
  \label{eq:defreg}
  \alt^{[\pm2]} \in \mathcal{C}^\infty \left(\mathbb{R}_t\times (-\infty,0]_{r^\star},\mathcal{S}^{[\pm2]}\right),
\end{align}
where $\mathcal{S}^{[\pm2]}$ is the space of smooth spin-$\pm2$-weighted complex functions on $\mathbb{S}^2$. 
Moreover, we say that $\alt^{[\pm2]}$ are \emph{regular at the future event horizon}, if for all $p,q\in\mathbb{N}$ and for all $t^\star_0\in\mathbb{R}$,
\begin{align}
  \label{eq:defreghor}
  L^p(\De^{-1}\Lb)^q\le(\De^{\pm1}\alt^{[\pm2]}\ri)\bigg|_{t^\star=t^\star_0} & = O_{r\to r_+}(1),
\end{align}
where $t^\star$ is the regular time function $t^\star := t + r^\star - \frac1k\arctan(kr)$, and where 
\begin{align*}
  L & := \Xi\pr_t + \frac{a\Xi}{r^2+a^2}\pr_\varphi + \pr_{r^\star}, & \Lb & := \Xi\pr_t+\frac{a\Xi}{r^2+a^2}\pr_\varphi -\pr_{r^\star}.  
\end{align*}
It is well-known (see~\cite{War15,Gra.Hol23}) that the Teukolsky equations~\eqref{eq:Teuk} admit a well-posed initial boundary value problem, provided that boundary conditions at infinity $(r^\star=0)$ are imposed. The natural boundary conditions, corresponding to prescribing the induced conformal metric at infinity to be the same as for anti-de Sitter space (see~\cite{Fri95,Gra.Hol23,Gra.Hol24}), read
\begin{align}
  \label{eq:TeukBC}
  \left(\alt^{[+2]} - \overline{\big(\alt^{[-2]}\big)}\right)\Big|_{r^\star=0} & = 0, & \left(\pr_{r^\star}\alt^{[+2]} + \pr_{r^\star}\overline{\big(\alt^{[-2]}\big)}\right)\Big|_{r^\star=0} & = 0.
\end{align}

A remarkable feature of Kerr-adS spacetimes is that the \emph{Hawking-Reall} Killing vector field
\begin{align*}
  \mathbf{K} := \pr_t + \om_+ \pr_\varphi, && \text{with} && \om_+ := \frac{a}{r_+^2+a^2},
\end{align*}
is globally causal on $\MM_{\mathrm{KadS}}$ if the \emph{Hawking-Reall bound}~\cite{Haw.Rea99} holds
\begin{align}\label{eq:HR}
  a & < k r_+^2.
\end{align}
This is in strong contrast with the asymptotically flat Kerr case, where, as soon as $a>0$, $\mathbf{K}$ becomes spacelike for $r$ sufficiently large.\\

In this article, we show that the Hawking-Reall bound~\eqref{eq:HR} is critical for the stability of Kerr-anti-de Sitter spacetimes as solutions to the Einstein equations~\eqref{eq:EEcosmo} (see discussions and a review of previous results at the end of this section). The following theorem is the main result of the paper.

\begin{theorem}\label{thm:main}
  Let $\Theta:s\in[0,1]\mapsto\left(M(s),a(s),k(s)\right)$ be a continuous path of admissible Kerr-adS parameters, such that $\left(M(0),a(0),k(0)\right)$ satisfy the Hawking-Reall bound~\eqref{eq:HR} and $\left(M(1),a(1),k(1)\right)$ strictly violate~\eqref{eq:HR}, \emph{i.e.}
  \begin{align*}
    a(0) < k(0)(r_+(0))^2 && \text{and} && a(1) > k(1)(r_+(1))^2.
  \end{align*}
  There exists $m_0\in\mathbb{N}$ depending on $\Theta$ such that for all $|m|\geq m_0$, there exists $s_m\in[0,1]$ and non-trivial stationary modes
  \begin{align}\label{eq:defstatmodes}
    \begin{aligned}
      \alt^{[+2]}(t,r,\varth,\varphi) & = e^{-im\om_+t} R^{[+2]}(r) S^{[+2]}(\varth) e^{+im\varphi}, \\
      \alt^{[-2]}(t,r,\varth,\varphi) & = e^{+im\om_+t} R^{[-2]}(r) S^{[-2]}(\varth) e^{-im\varphi},
    \end{aligned}
  \end{align}
  which are regular spin-$\pm2$-weighted complex-valued functions~\eqref{eq:defreg}, solutions to the Teukolsky equations~\eqref{eq:Teuk} with $(M,a,k)=\left(M(s_m),a(s_m),k(s_m)\right)$, satisfying the regularity conditions~\eqref{eq:defreghor} at the horizon and the boundary conditions~\eqref{eq:TeukBC} at infinity. Moreover, we have
  \begin{align}\label{eq:limHRminfinity}
    a(s_m) - k(s_m)(r_+(s_m))^2 & \xrightarrow{|m|\to+\infty} 0,
  \end{align}
  \emph{i.e.} the black hole parameters $\left(M(s_m),a(s_m),k(s_m)\right)$ accumulate at the Hawking-Reall threshold when $|m|\to+\infty$.
\end{theorem}

\paragraph{Remarks on Theorem~\ref{thm:main}.}
\begin{enumerate}
\item The modes~\eqref{eq:defstatmodes} are called \emph{stationary} since $\mathbf{K}(\alt^{[+2]}) = \mathbf{K}(\alt^{[-2]}) = 0$. The existence of such modes shows that there exist non-decaying in time solutions to the Teukolsky equations~\eqref{eq:Teuk}. Moreover, it is strongly suspected to be the offset of a stronger exponential instability, see discussions in the next paragraph.
\item Our result does not say that the parameters $\left(M(s_m),a(s_m),k(s_m)\right)$, for which there exists non-trivial stationary mode solutions, violate the Hawking-Reall bound~\eqref{eq:HR}. This is nevertheless strongly suspected given the numerics in~\cite{Car.Dia.Har.Leh.San14}. 
\item Using the so-called ``metric reconstruction method'' (see for example~\cite{Dia.Rea.San09}), we claim that one can associate to the modes $\alt^{[\pm2]}$ of Theorem~\ref{thm:main} non-trivial stationary solutions to the linearisation of the full Einstein equations~\eqref{eq:EEcosmo} around the Kerr-adS metrics.   
\end{enumerate}

In this paper, we also prove the following high-frequency mode stability result (see also~\cite{Gra.Hol23} for similar results). 
\begin{theorem}\label{thm:modestab}
  For all Kerr-adS admissible parameters $(M,a,k)$ which satisfy the Hawking-Reall bound~\eqref{eq:HR}, there exists $m_0\in\mathbb{N}$ depending on $(M,a,k)$ such that for all $|m|\geq m_0$, there does not exist non-trivial stationary modes~\eqref{eq:defstatmodes} which are regular spin-$\pm2$-weighted complex-valued functions~\eqref{eq:defreg}, solutions to the Teukolsky equations~\eqref{eq:Teuk}, satisfying the regularity conditions~\eqref{eq:defreghor} at the horizon and the boundary conditions~\eqref{eq:TeukBC} at infinity.
\end{theorem}
Theorem~\ref{thm:modestab} is proved in Section~\ref{sec:proofthmmodestab}.

\paragraph{The Hawking-Reall bound and the stability of Kerr-adS spacetimes.}
It was first argued in~\cite{Haw.Rea99} that the Hawking-Reall bound~\eqref{eq:HR} being satisfied or not is expected to have strong consequences on the stability of Kerr-adS spacetimes as solutions to the Einstein equations~\eqref{eq:EEcosmo}. In this paragraph we review previous results or arguments supporting this expectation.\\

First consider the conformal wave equation with Dirichlet boundary conditions
\begin{align}\label{eq:wave}
  0 & = \Box_{g_{\mathrm{KadS}}} \psi + 2k^2\psi, & \psi\big|_{r^\star=0} = 0, 
\end{align}
which can be seen as a toy model for the Teukolsky equations~\eqref{eq:Teuk} with boundary conditions~\eqref{eq:TeukBC}. It has been proved in~\cite{Hol09,Hol.War14} that all solutions to~\eqref{eq:wave} remain uniformly bounded, as long as the Hawking-Reall bound~\eqref{eq:HR} holds. Upon additional assumptions, it has also been shown that the solutions to~\eqref{eq:wave} decay in time~\cite{Hol.Smu13}. These results rely on the existence of the globally causal Killing field $\mathbf{K}$ and of a well-behaved energy-momentum tensor for the wave equation~\eqref{eq:wave}. \emph{A contrario}, for Kerr-adS spacetimes such that $a>kr_+^2$, stationary and (exponentially) growing in time solutions to~\eqref{eq:wave} have been constructed in~\cite{Dol17}. We will discuss the proof of~\cite{Dol17} in the next paragraph. Let us also say that exponentially growing modes were constructed in~\cite{Zhe24a} for Klein-Gordon equations on Reissner-Nordström-adS spacetimes.\\

For the Teukolsky equations~\eqref{eq:Teuk}, the techniques used for the wave equation~\eqref{eq:wave} do not directly apply by lack of a well-behaved energy-momentum tensor (in rough terms, this is due to the first order derivatives in~\eqref{eq:Teuk}). Nevertheless, it was proved in~\cite{Gra.Hol24a} that all solutions to~\eqref{eq:Teuk} on Schwarzschild-adS ($a=0$) are bounded and decay in time, which was used in~\cite{Gra.Hol24} to show that Schwarzschild-adS is linearly stable as a solution to the Einstein equations~\eqref{eq:EEcosmo}. This relies on the so-called \emph{Chandrasekhar transformations} which give a correspondence between solutions to the Teukolsky equations and solutions to conformal wave equations with potential known as the \emph{Regge-Wheeler equations}, for which the techniques used for the wave equation~\eqref{eq:wave} can be adapted. On rotating Kerr-adS spacetimes ($a\neq0$) such an (exact) correspondence fails to hold and the only rigorous stability result is mode stability under some additional hypothesis on the black hole parameters~\cite{Gra.Hol23}.\\

On the side of instability, it has been argued in~\cite{Car.Dia04,Car.Dia.Lem.Yos04,Car.Dia.Yos06} that when the Hawking-Reall bound is violated (and under slow rotation and small radius hypothesis), exponentially growing solutions for the Teukolsky equations should exist. Numerical evidences were given in~\cite{Car.Dia.Har.Leh.San14}. Let us highlight one feature of~\cite{Car.Dia.Har.Leh.San14} which was later proved in~\cite{Gra.Hol23} and is particularly relevant for the present paper: at the offset of instability, \emph{i.e.} for Kerr-adS parameters ``just before'' exponentially growing solutions to~\eqref{eq:Teuk} are numerically shown to exist, there should exist stationary modes~\eqref{eq:defstatmodes} solutions to~\eqref{eq:Teuk}. From the existence of these stationary modes, it is expected that exact solutions to the Einstein equations~\eqref{eq:EEcosmo} with a Killing field bifurcating from the Kerr-adS solutions (called ``black resonators'') should exist~\cite{Dia.San.Way15} (see also~\cite{Cho.Shl17} and~\cite{Zhe24} for rigorous proofs of this type of result on Kerr and on Reissner-Nordström-adS spacetimes respectively).\\

Finally, we mention that~\cite{Gre.Hol.Ish.Wal16} have shown that for all Kerr-adS spacetimes such that $a > kr_+^2$, there exists solutions to the linearisation of the Einstein equations~\eqref{eq:EEcosmo} which do not decay in time. The result of~\cite{Gre.Hol.Ish.Wal16} is based on the conservation/decay of a canonical energy for the linearised Einstein equations~\eqref{eq:EEcosmo}.\footnote{The canonical energy does not come from an energy momentum tensor but from a symplectic form. Hence it is less clear how to use it to prove quantitative boundedness or decay results. See discussions and results in that direction in~\cite{Col23}.} If the Hawking-Reall bound is violated, perturbations can be constructed so that this canonical energy is initially negative, hence the canonical energy cannot decay to $0$ at large times. 
Note that this result says that Kerr-adS spacetimes are linearly unstable solutions to the Einstein equations~\eqref{eq:EEcosmo} but does not give the nature of the instability, \emph{i.e.} whether there exist stationary, oscillating or exponentially growing solutions to the linearisation of~\eqref{eq:EEcosmo}.


\paragraph{The wave equation \emph{vs} Teukolsky equations.}
In this paragraph we review the method of proof used in~\cite{Dol17} to construct stationary solutions to the wave equation~\eqref{eq:wave} on Kerr-adS spacetimes.\footnote{Other wave equations and other boundary conditions were also considered in~\cite{Dol17}. It is further shown in~\cite{Dol17} that the existence of exponentially growing solutions to~\eqref{eq:wave} can be inferred from the existence of stationary modes.}  We mention that the proof in~\cite{Dol17} adapts arguments of the seminal~\cite{Shl14} and that similar arguments were also used in~\cite{Col21},~\cite{Zhe24a}.\\

First, with the stationary mode ansatz
\begin{align}
  \label{eq:wavemodeansatz}
  u(t,r,\varth,\varphi)=e^{-im\om_+ t}R(r)S(\varth)e^{+im\varphi},
\end{align}
the wave equation~\eqref{eq:wave} separates into a radial and an angular ODE
\begin{align}
  0 & = - \pr^2_{r^\star}R + w\le(\lambda-\gamma m^2+V_0\ri)R, \label{eq:radialwave}\\
  0 & = - \frac{1}{\sin\varth}\pr_\varth\le(\De_\varth\sin\varth\pr_\varth S\ri) + \le(G_0 + G_1m^2 - \lambda\ri) S,\label{eq:angularwave}
\end{align}
where $\la$ is the separation constant, where $w,V_0:(-\infty,0)\to\mathbb{R}$ and $G_0,G_1:(0,\pi)\to\mathbb{R}$ are smooth functions, depending only on the black hole parameters $(M,a,k)$, and where
\begin{align*}
  \gamma & := \Xi^2 + \Xi^2\frac{(kr_+^2-a)(kr_+^2+a)}{k^2(r_+^2+a^2)^2}.
\end{align*}
Regularity at the horizon and the Dirichlet boundary conditions at infinity~\eqref{eq:wave} impose that
\begin{align}
  R \quad \text{extends as a smooth function at $r=r_+$,}  && R(r^\star=0)  = 0, \label{eq:limitswaveR}
\end{align}
and regularity at the poles $\varth=0,\pi$, imposes that
\begin{align}
   S(0) & = S(\pi) = 0.\label{eq:limitswaveS}
\end{align}
Equation~\eqref{eq:angularwave} with~\eqref{eq:limitswaveS} defines an eigenvalue problem with boundary conditions. This eigenvalue problem admits a fundamental eigenvalue and eigenfunction $\la_m\in\mathbb{R}$ and $S_m:(0,\pi)\to\mathbb{R}$ which realise the infimum of the Rayleigh quotient
\begin{align}\label{eq:Rayleighlambda}
  \la_m & = \inf_{\int|S|^2\sin\varth\d\varth=1} \int_{0}^\pi\le(\De_\varth\le|\pr_{\varth}S\ri|^2+(G_0+G_1m^2)|S|^2\ri)\,\sin\varth\d\varth.
\end{align}
Choosing $S=S_{m}$ in the ansatz~\eqref{eq:wavemodeansatz}, the goal is to solve~\eqref{eq:radialwave} with $\la=\la_{m}$. To that end, the trick is to consider an auxiliary eigenvalue problem:
\begin{align}\label{eq:auxeigenvalue}
  \si R & = - \pr^2_{r^\star}R + w\le(\lambda_m-\gamma m^2+V_0\ri)R,
\end{align}
with boundary conditions~\eqref{eq:limitswaveR}. This eigenvalue problem admits a fundamental eigenvalue and eigenfunction $\si_m\in\mathbb{R}$ and $R_m:(-\infty,0)\to\mathbb{R}$ which realise the infimum of the Rayleigh quotient
\begin{align}\label{eq:Rayleighsigma}
  \si_m & = \inf_{\int|R|^2\d r^\star=1} \int_{-\infty}^0\le(\le|\pr_{r^\star}R\ri|^2+w\le(\lambda_m-\gamma m^2+V_0\ri)|R|^2\ri)\,\d r^\star.
\end{align}
Using the minimisation property~\eqref{eq:Rayleighlambda} and that the minimum of $G_1$ on $(0,\pi)$ is $\Xi^2$, we have the semi-classical limit\footnote{See the proof of Lemma~\ref{lem:angular} for a justification of this type of result.}
\begin{align*}
  \frac{\la_m}{m^2} \xrightarrow{|m|\to+\infty} \Xi^2.
\end{align*}
Thus, we have the asymptotic
\begin{align*}
  \lambda_m-\gamma m^2+V_0 \sim m^2\Xi^2\frac{(kr_+^2-a)(kr_+^2+a)}{k^2(r_+^2+a^2)^2}, && \text{when $|m|\to+\infty$,}
\end{align*}
\emph{i.e.} the positivity of the potential in the Rayleigh quotient~\eqref{eq:Rayleighsigma} depends on the Hawking-Reall bound~\eqref{eq:HR} being satisfied or not. If~\eqref{eq:HR} holds the potential is positive and one has $\si_m>0$ for $|m|$ sufficiently large. \emph{A contrario}, if~\eqref{eq:HR} is strictly violated, \emph{i.e.} $a>kr_+^2$, a semi-classical argument gives that $\si_m<0$ for $|m|$ sufficiently large.
Thus, along a continuous path of admissible Kerr-adS parameters $\le(M(s),a(s),k(s)\ri)$ such that the Hawking-Reall bound holds at one end $a(0)<k(0)(r_+(0))^2$ and is violated at the other $a(1)>k(1)(r_+(1))^2$, there exists parameters $\le(M(s_m),a(s_m),k(s_m)\ri)$ such that $\si_m=0$ provided that $|m|$ is sufficiently large depending on $\le(M(0),a(0),k(0)\ri)$ and $\le(M(1),a(1),k(1)\ri)$. In that case, the mode ansatz~\eqref{eq:wavemodeansatz} defines a regular solution to the wave equation with boundary condition~\eqref{eq:wave}, and this concludes the argument of~\cite{Dol17}.\\

We now comment on the main differences with the Teukolsky case.
\begin{enumerate}
\item The first obvious difference is that there is not one but two (one for each spin $[+2]$ and $[-2]$) radial and angular ODEs corresponding to~\eqref{eq:radialwave} and \eqref{eq:angularwave}, and that solutions to the two $[\pm2]$ radial equations have to be matched at $r^\star=0$ using the boundary conditions~\eqref{eq:TeukBC}.
\item The second -- and most critical -- difference is that, due to the first order terms in the Teukolsky equation~\eqref{eq:Teuk}, the potentials and the unknown $R$ in the radial equations corresponding to~\eqref{eq:radialwave} are \underline{complex-valued} (see Section~\ref{sec:prelim} for exact formulas). Hence one cannot define a real-valued spectral parameter $\si_m$ and apply the above discussed continuity argument of~\cite{Dol17}, which critically relies on the intermediate value theorem. 
\end{enumerate}



\paragraph{Overview of the proof of the main theorem.}
The proof of Theorem~\ref{thm:main} decomposes along the following steps.
\begin{enumerate}
\item The first step (Section~\ref{sec:shooting}) is to show that one can reduce the two radial Teukolsky ODE in the $[\pm2]$ cases to one radial ODE for a single quantity $R$ and that, defining
  \begin{align*}
    \wp[R](r^\star) & := 
                      \pr_{r^\star}\le(\le|R(r^\star)\ri|^2\ri),
  \end{align*}
  the boundary conditions for $R$ become
  \begin{align}\label{eq:BCoverview}
    \De^{-1}R~\text{extends smoothly at $r=r_+$} && \text{and} && \wp[R](0) =0.
  \end{align}
  The crucial point is that $\wp[R]:(-\infty,0)\to\mathbb{R}$ is manifestly real-valued and can therefore be used as a ``bullet'' quantity in a continuity argument, the ``target'' boundary condition being $\wp[R](0) =0$. In the rest of the argument we ``shoot'' $R$ from $r=r_+$, \emph{i.e.} we always consider the unique (normalised) ODE solution $R$ such that $\De^{-1}R$ extends smoothly at $r=r_+$. In particular, this implies that $\wp[R](r^\star) \to 0$ when $r^\star\to-\infty$.
\item The second step is to show that
  \begin{itemize}
  \item $\wp[R](r^\star)$ is always positive close to $r^\star=-\infty$,
  \item $\wp[R](r^\star)$ stays positive for all $r^\star\in(-\infty,0]$ if the Hawking-Reall bound~\eqref{eq:HR} is satisfied and $|m|$ is sufficiently large depending on $(M,a,k)$.
  \end{itemize}
  \emph{En passant}, one sees that the last boundary condition~\eqref{eq:BCoverview} cannot be satisfied and this directly proves the mode stability of Theorem~\ref{thm:modestab} (see Section~\ref{sec:proofthmmodestab}). 
\item The third step is to show that if the Hawking-Reall bound~\eqref{eq:HR} is strictly violated, there exists $r^\star_0\in(-\infty,0]$ such that $\wp[R](r^\star_0)<0$, provided that $|m|$ is sufficiently large. This is done in Section~\ref{sec:envelope} and uses a WKB approximation close to $r^\star=0$. In fact, we prove that for \emph{any} solution to the radial Teukolsky ODE, $\wp[R]$ becomes strictly negative somewhere close to $r^\star=0$. This is delicate since the first term of the WKB approximation close to $r^\star=0$ are the pure waves
  \begin{align}\label{eq:purewave}
    e^{+i\omo r^\star} && \text{and} && e^{-i\omo r^\star},
  \end{align}
  where $\omo$ is a (high) frequency parameter (see Section~\ref{sec:WKB}), for which one has $\wp = 0$ identically. Hence one has to compute the next term (Section~\ref{sec:nointerferences}) to show that $\wp<0$ (uniformly close to $r^\star=0$) for each of the WKB approximation corresponding to the pure waves~\eqref{eq:purewave}. For non-trivial linear combinations of the pure waves~\eqref{eq:purewave}, one has interference and $\wp$ takes negative values at some point (Section~\ref{sec:proofenvelope}). 
\item The conclusion is as follows: For a continuous path of parameters $s\in[0,1]\mapsto(M(s),a(s),k(s))$, such that $a(0)<k(0)(r_+(0))^2$ and $a(1)>k(1)(r_+(1))^2$, one considers the first parameter $s_m>0$ such that $\wp[R]$ is not strictly positive on the whole domain $(-\infty,0)$. This means that $\wp[R]$ must vanish somewhere on $(-\infty,0]$ and the goal is to prove that, at that parameter $s_m$, this vanishing must occur at the end of the domain, \emph{i.e.} that the desired boundary condition $\wp[R](0)=0$ holds. To that end, we show in Section~\ref{sec:norebound} that the ``bullet'' $\wp[R]$ cannot ``re-bound'' on $\wp=0$, \emph{i.e.} that if $\wp[R]=\wp[R]'=0$ at $r^\star_0$, either $R$ is trivial or $\wp[R]''<0$ at $r^\star_0$, provided that $|m|$ is sufficiently large. Here it is crucial that the estimates are \emph{uniform} for all parameters on the path, \emph{i.e.} that they do not break down when the Hawking-Reall threshold is crossed.
\end{enumerate}
Note that Steps 2 to 4 all require bounds on the angular eigenvalues $\la$ which are proved in Section~\ref{sec:angular}. The above scheme of proof is formalised in Section~\ref{sec:proofthmmain}.

\paragraph{Notations.} We write ``$C=C(a,b)>0$'' for ``there exists a constant $C$ depending only on $a,b$ and universal constants'' and ``$A\les_{a,b}B$'' for ``$A \leq C B$ with $C=C(a,b)>0$''.

\paragraph{Acknowledgements.}
I thank Dietrich H\"afner for many interesting discussions around this project. I thank Gustav Holzegel for introducing me to this subject, re-reading and comments.

\section{The decoupled radial Teukolsky problem}\label{sec:shooting}
\subsection{The separated Teukolsky equations}\label{sec:prelim}
In this section, we reduce our results to ODE analysis, using the separability property of Teukolsky equations. We will use the following oblate spheroidal harmonics basis for spin-$\pm2$-weighted complex functions.
\begin{lemma}\label{lem:hilbert}
  There exists an Hilbert basis $\left(S_{m\ell}(\varth)e^{+im\varphi}\right)_{m\in\mathbb{Z},\ell\geq |m|} $ of the smooth spin-$+2$-weighted complex functions, with associated real eigenvalues $\left(\la_{m\ell}\right)_{\ell\geq 2,\ell\geq |m|}$ given by\footnote{Equation~\eqref{eq:defangla} coincides with~\cite[(3.20), (3.22)]{Kha83}.}
  \begin{align}\label{eq:defangla}
    \begin{aligned}
    -\la_{m\ell}S_{m\ell} & = \sqrt{\De_\varth}\LL_{-1}^\dg\sqrt{\De_\varth}\LL_2S_{m\ell}+\le(-6a(\Xi\om_+m)\cos\varth+6a^2k^2\cos^2\varth\ri)S_{m\ell},
    \end{aligned}
  \end{align}
  where
  \begin{align*}
    \LL_{-1}^\dg & := \pr_\varth + \frac{\Xi H}{\De_\varth} - \pr_\varth\le(\log\le(\sqrt{\De_\varth}\sin\varth\ri)\ri), & \LL_2 & := \pr_\varth - \frac{\Xi H}{\De_\varth} +2 \pr_\varth\le(\log\le(\sqrt{\De_\varth}\sin\varth\ri)\ri),
  \end{align*}
  and
  \begin{align*}
     H & := m \le(a\om_+\sin\varth-\frac1{\sin\varth}\ri).
  \end{align*}
  We sort the eigenvalues $\la_{m\ell}$ so that, for all $m\in\mathbb{Z}$, the sequence $(\la_{m\ell})_{\ell\geq \max(2,|m|)}$ is increasing.
\end{lemma}
\begin{remark}\label{rem:spin-2}
  The functions $S_{m\ell}$ from Lemma~\ref{lem:hilbert} are real-valued and $\left(S_{m\ell}(\varth)e^{-im\varphi}\right)_{m\in\mathbb{Z},\ell\geq |m|}$ defines an Hilbert basis of the smooth spin-$-2$-weighted complex functions.
\end{remark}
\begin{remark}\label{rem:hilbertbis}
  The eigenvalue equation~\eqref{eq:defangla} rewrites as
  \begin{align}\label{eq:defanglabis}
    \begin{aligned}
      -\la_{m\ell}S_{m\ell} & = \frac{1}{\sin\varth}\pr_\varth\le(\De_\varth\sin\varth\pr_\varth S_{m\ell}\ri) - G_m(\varth) S_{m\ell},
    \end{aligned}
  \end{align}
  with
  \begin{align*}
    G_m(\varth) & := \frac{\Xi^2H^2}{\De_\varth} +2\De_\varth +2a^2k^2 +4 \Xi^2\frac{\cot^2\varth}{\De_\varth} - 4 \frac{\Xi^2 H\cot\varth}{\De_\varth} +8 a^2k^2\frac{\Xi m\cos\varth}{\De_\varth} + 8 \Xi(\Xi a m\om_+)\cos\varth. 
  \end{align*}
   In particular, for all $|m|\geq 2$, $\la_{m|m|}$ is the fundamental eigenvalue and satisfies the minimisation property
  \begin{align}\label{eq:minimisation}
    \la_{m|m|} & = \inf \le( \int_{0}^\pi\le(\De_\varth|\pr_\varth S|^2 + G_m(\varth)|S|^2\ri)\sin\varth\d\varth\ri). 
  \end{align}
  where the infimum is taken on all $S\in C^\infty_c((0,\pi),\mathbb{R})$ such that $\int_0^\pi |S|^2\sin\varth\d\varth=1$.
\end{remark}

We define the following \emph{stationary radial Teukolsky equation with separation constant $\la$}:\footnote{Equation~\eqref{eq:radialstatTeuk} is a rewriting of~\cite[(3.19), (3.21)]{Kha83} with the normalisation~\eqref{eq:renormalt}.}
\begin{align}\label{eq:radialstatTeuk}
  0 & = -R'' + V[\la]R,
\end{align}
where we denote by $'$ derivation with respect to $r^\star$, and where $V[\la]=V_0[\la]+V_{00}-iV_1$ with
\begin{align*}
  \begin{aligned}
    V_0[\la] & := \frac{\De}{(r^2+a^2)^2}\le(\la-2-a^2k^2\ri) - \le(\Xi\om_+m\ri)^2\frac{(r-r_+)^2(r+r_+)^2}{(r^2+a^2)^2} \\
    & = \frac{\De}{(r^2+a^2)^2}\lat + \le(\frac{\Xi\om_+m}{k}\ri)^2\frac{\De-k^2(r-r_+)^2(r+r_+)^2}{(r^2+a^2)^2} \\
    V_{00} & := \frac{(\pr_r\De)^2}{(r^2 + a^2)^2} + \frac{\De}{(r^2+a^2)^2}\le(2+a^2k^2-6k^2r^2-\pr_r^2\De\ri) +r\frac{\De\pr_r\De}{(r^2+a^2)^{3}} - (2r^2-a^2)\frac{\De^2}{(r^2+a^2)^4} \\
    V_1 & := \frac{\pr_r\Delta}{(r^2+a^2)^2}2(\Xi\om_+m) \left(r_+^2-r^2\right) + \frac{\Delta}{(r^2+a^2)^2}8\Xi \om_+m r,
  \end{aligned}
\end{align*}
with
\begin{align*}
  \lat := \la - 2-a^2k^2-\left(\frac{\Xi\om_+m}{k}\right)^2.
\end{align*}
\begin{remark}
  The potentials $V_0[\la], V_{00}, V_1:(r_+,+\infty)_r\to\mathbb{R}$ extend as smooth functions both on the horizon, \emph{i.e.} as functions on $[r_+,+\infty)_r$ and at infinity, \emph{i.e.} as functions on $(-\infty,0]_{r^\star}$. In particular, $r^\star=0$ is an ordinary regular point for the ODE~\eqref{eq:radialstatTeuk}. 
\end{remark}

We have the following lemma.
\begin{lemma}\label{lem:sep}
   For all $m\in\mathbb{Z},\ell\geq |m|$, the maps 
  \begin{align}\label{eq:statansatz}
    \begin{aligned}
      \alt^{[+2]}(t,r,\varth,\varphi) & = e^{-im\om_+ t}R^{[+2]}(r) S_{m\ell}(\varth) e^{+im\varphi}, \\
      \alt^{[-2]}(t,r,\varth,\varphi) & = e^{+im\om_+ t} R^{[-2]}(r) S_{m\ell}(\varth) e^{-im\varphi},
    \end{aligned}
  \end{align}
  are smooth (up to infinity) spacetime spin-$\pm2$-weighted complex functions~\eqref{eq:defreg} satisfying the Teukolsky equations~\eqref{eq:Teuk} if and only if $R^{[\pm2]}$ are smooth complex-valued functions on $(-\infty,0]_{r^\star}$, solutions to the (same) radial Teukolsky equation~\eqref{eq:radialstatTeuk} with separation constant $\la=\la_{m\ell}$. Moreover,
  \begin{itemize}
  \item $\alt^{[\pm2]}$ are regular at the horizon in the sense of~\eqref{eq:defreghor} iff there exists smooth functions $F^{[\pm2]}:[r_+,+\infty)\to\CCC$ such that 
    \begin{align}
      \label{eq:defreghorradialoriginal}
      \De^{\pm1}(r) R^{[\pm2]}(r) & =  F^{[\pm2]}(r),
    \end{align}
    for all $r\in(r_+,+\infty)$,
  \item $\alt^{[\pm2]}$ satisfy the boundary conditions at infinity~\eqref{eq:TeukBC} iff
    \begin{align}\label{eq:TeukBCradial}
      \left(R^{[+2]} - \overline{R^{[-2]}}\right)\Big|_{r^\star=0} & = 0, & \left(\left(R^{[+2]}\right)' + \left(\overline{R^{[-2]}}\right)'\right)\Big|_{r^\star=0} & = 0. 
    \end{align}
  \end{itemize}
\end{lemma}
\begin{proof}
  The equivalence between the Teukolsky equations~\eqref{eq:Teuk} and the radial Teukolsky equations~\eqref{eq:radialstatTeuk}, for modes of the form~\eqref{eq:statansatz}, is a direct computation which is left to the reader (see~\cite{Kha83}). The fact that the regularity condition at the horizon for $\alt^{[\pm2]}$ \eqref{eq:defreghor} and for $R^{[\pm2]}$ \eqref{eq:defreghorradialoriginal} are equivalent follows from the fact that for stationary modes~\eqref{eq:statansatz}, 
  \begin{align*}
    \De^{-1}\Lb & = -\le(a\Xi m \frac{r_+^2+a^2}{r^2+a^2} \frac{(r^2-r_+^2)}{\De}+\frac{\pr_r}{r^2+a^2}\ri),
  \end{align*}
  and that $\frac{(r^2-r_+^2)}{\De}$ is a smooth function at $r=r_+$, \emph{i.e.} $\De^{-1}\Lb$ is a regular first order operator at $r=r_+$. The equivalence between the boundary conditions~\eqref{eq:TeukBC} and~\eqref{eq:TeukBCradial} is immediate.
\end{proof}

\subsection{The shooting ``bullet'' and ``target''}\label{sec:reduction}
The following lemma is the key to our continuity argument.
\begin{lemma}\label{lem:target}
  Let $R$ be a smooth complex-valued function on $(-\infty,0]_{r^\star}$, solution to the radial Teukolsky equation~\eqref{eq:radialstatTeuk}, regular at the horizon in the sense that there exists a smooth function $F:[r_+,+\infty)\to\mathbb{C}$ such that 
  \begin{align}\label{eq:defreghorradial}
    R(r) & = \De(r) F(r),
  \end{align}
  for all $r > r_+$, and which satisfies the ``target'' boundary condition
  \begin{align}\label{eq:BCstat}
    \left(\bar{R}R' + R\bar{R}'\right)\Big|_{r^\star=0} = 0.
  \end{align}
  Then, there exists $\kappa\in\mathbb{C}^\ast$ such that the maps $R^{[+2]} := R$ and $R^{[-2]}:=\kappa R$ are solutions to the radial Teukolsky equation~\eqref{eq:radialstatTeuk}, regular at the horizon in the sense of~\eqref{eq:defreghorradialoriginal} and satisfy the coupled boundary conditions~\eqref{eq:TeukBCradial}.
\end{lemma}
\begin{proof}
  Let $\kappa\in\mathbb{C}^\ast$ a constant which will be determined in the sequel and define $R^{[+2]} := R$ and $R^{[-2]}:=\kappa R$. Since $R$ is regular at the horizon in the sense of~\eqref{eq:defreghorradial}, $R^{[-2]}$ is regular in the sense of~\eqref{eq:defreghorradialoriginal} and $R^{[+2]}$ is (even more) regular than required in~\eqref{eq:defreghorradialoriginal}. Now we show that if~\eqref{eq:BCstat} holds, one can recover the boundary condition~\eqref{eq:TeukBCradial} by an appropriate choice of $\kappa$. If $R=0$ there is nothing to do and one can take any $\kappa\in\mathbb{C}^\ast$. If else and if $R(0)=0$, one can simply define $\kappa := -\frac{\bar{R}'(0)}{R'(0)}$ and the conditions~\eqref{eq:TeukBCradial} are satisfied (note that by Cauchy-Lipschitz one cannot have $R'(0)=0$ because $R\neq0$). If else, one defines $\kappa := \frac{\bar{R}(0)}{R(0)}$ so that $R^{[+2]}-\overline{R^{[-2]}} = R - \bar\kappa \bar R = 0$ at $r^\star=0$, which is the first boundary condition of~\eqref{eq:TeukBCradial}. Then, one checks that~\eqref{eq:BCstat} gives
  \begin{align*}
    \left(R^{[+2]}\right)' + \left(\overline{R^{[-2]}}\right)' & = \frac{1}{\bar{R}}\left(\bar{R}R'+ R\bar R'\right) + \left(\bar\kappa-\frac{R}{\bar{R}}\right) \bar R' = 0
  \end{align*}
  at $r^\star=0$, and the second boundary condition of~\eqref{eq:TeukBCradial} is satisfied, as desired.
\end{proof}


We have the following reciproque (see~\cite{Gra.Hol23}).
\begin{lemma}\label{lem:reciproque}
  For all solutions $(R^{[+2]},R^{[-2]})$ of the Teukolsky equations~\eqref{eq:radialstatTeuk} which are regular at the horizon~\eqref{eq:defreghorradialoriginal} and satisfy the coupled boundary condition~\eqref{eq:TeukBCradial}, then $R:=R^{[+2]}$ is regular at the horizon~\eqref{eq:defreghorradial} and satisfies the ``target'' boundary condition~\eqref{eq:BCstat}.
\end{lemma}
\begin{proof}
  Denoting by $\mathfrak{W}:=R^{[+2]}\le(R^{[-2]}\ri)'-R^{[-2]}\le(R^{[+2]}\ri)'$ the Wronskian between $R^{[+2]}$ and $R^{[-2]}$ (which are both solutions to the same ODE~\eqref{eq:radialstatTeuk}), the coupled boundary conditions~\eqref{eq:TeukBCradial} imply that $\mathfrak{W} = -\left(\bar{R}R' + R\bar{R}'\right)$ at $r^\star=0$, with $R:=R^{[+2]}$. Thus, the ``target'' boundary condition~\eqref{eq:BCstat} is satisfied provided that $R^{[+2]}$ and $R^{[-2]}$ are colinear. A Frobenius analysis of the ODE~\eqref{eq:radialstatTeuk} at the regular singular point $r=r_+$ shows that~\eqref{eq:radialstatTeuk} has indicial roots $+1$ and $-1$, hence has two branches of solutions
  \begin{align*}
    R_- & = \De F_-, & R_{+} & = \De^{-1}F_+ + C \De \log(\De) F_-,
  \end{align*}
  with $F_+,F_-:[r_+,+\infty)\to\mathbb{C}$ two smooth functions normalised by $F_{+}(r_+)=F_-(r_+)=1$ and where $C$ is a constant depending of the parameters. Using the Teukolsky-Starobinsky identities and positivity of the Teukolsky-Starobinsky constants, one can show that $C\neq0$ for all admissible sub-extremal Kerr-adS parameters, see~\cite[erratum]{Gra.Hol23}. $R^{[-2]}$ being regular at the horizon in the sense of~\eqref{eq:defreghorradialoriginal} implies that $R^{[-2]}$ is colinear to $R_-$. Since $C\neq0$, the branch $R_{+}$ does not satisfy the regularity required for $R^{[+2]}$ because $\De R_+ = F_+ + C\De^2 \log\De F_- \notin C^2([r_+,+\infty))$. Thus, $R^{[+2]}$ must be colinear to $R_-$ hence to $R^{[-2]}$ and this finishes the proof of the lemma. 
\end{proof}

In view of Lemmas~\ref{lem:target},~\ref{lem:reciproque}, we define the following quantities:
\begin{align*}
  \wp[R] & := 2 \Re(R\bar R') = R\bar R' + \bar R R' = \left(|R|^2\right)' \in\mathbb{R}, & W[R] & := 2i\Im(R\bar R') = R\bar R' - R'\bar R \in i\mathbb{R}.
\end{align*}
Direct computations give the following lemma.
\begin{lemma}
  For all solution $R$ to the radial Teukolsky ODE~\eqref{eq:radialstatTeuk}, we have
  \begin{subequations}\label{eq:envelopewronskien}
    \begin{align}
      \wp' & = 2 |R'|^2 + 2 \Re(V) |R|^2, \label{eq:wpderiv1}\\
      \wp'' & = 2i\Im(V) W + 2\Re(V')|R|^2 + 4\Re(V) \wp, \label{eq:wpderiv2}
    \end{align}
    and
    \begin{align}\label{eq:wronskien}
      W' & = -2i\Im(V)|R|^2 = 2iV_1 |R|^2.
    \end{align}
  \end{subequations}
\end{lemma}



\section{The angular eigenvalues and the Hawking-Reall threshold}\label{sec:angular}
This section is dedicated to the proof of the following (uniform) limit for the angular eigenvalues.
\begin{lemma}\label{lem:angular}
  For all compact subset of admissible Kerr-adS parameters $K$, we have
  \begin{align}\label{lim:lam}
    \frac{\la_{m|m|}}{m^2} \xrightarrow{|m|\to+\infty} \Xi^2\le(1-a\om_+\ri)^2 = \Xi^2 \le(\frac{\om_+}{k}\ri)^2 + \Xi^2\frac{kr_+^2-a}{k(r_+^2+a^2)} \frac{kr_+^2+a}{k(r_+^2+a^2)}, 
  \end{align}
  for all $(M,a,k)$ in $K$ and the convergence is uniform on $K$.
\end{lemma}
\begin{proof}
  Inspecting the potential $G_m$ defined in~\eqref{eq:defanglabis}, we have
  \begin{align}\label{eq:CVUG}
    \widetilde{G}_m(\varth) & := \frac{\De_\varth\sin^2(\varth)}{\Xi^2} \le(\frac{G_m(\varth)}{m^2} - \frac{\Xi^2H^2}{\De_\varth m^2}\ri) \xrightarrow{|m|\to+\infty} 0, 
  \end{align}
  uniformly for all $\varth\in(0,\pi)$ and for all $(M,a,k)\in K$. Hence, for all $0<\de<\min_{K}(1-a\om_+)^2$ (for all admissible parameters, one has $0\leq a\om_+<1$), there exists $m_0=m_0(K)>0$ such that
  \begin{align*}
    \begin{aligned}
      \frac{G_m(\varth)}{m^2} & = \frac{\Xi^2}{\De_\varth \sin^2\varth}\le(1-a\om_+\sin^2\varth\ri)^2 + \frac{\Xi^2}{\De_\varth \sin^2\varth} \widetilde{G}_m(\varth) \\
      & \geq \frac{\Xi^2}{\De_\varth\sin^2\varth}\le(\le(1-a\om_+\ri)^2 - \de\ri) \\
      & \geq \Xi^2 \le(\le(1-a\om_+\ri)^2 - \de\ri), 
    \end{aligned}
  \end{align*}
  for all $|m|\geq m_0$ for all $\varth\in(0,\pi)$ and for all $(M,a,k)\in K$. Using the variational property~\eqref{eq:minimisation}, we obtain that
  \begin{align}\label{est:lowerboundlaproof}
    \frac{\la_{m|m|}}{m^2} & \geq \Xi^2 \le(\le(1-a\om_+\ri)^2 - \de\ri),
  \end{align}
  for all $|m|\geq m_0$ and for all $(M,a,k)\in K$.\\

  Let us now prove an upper bound for $\la_{m|m|}$. Let $h>0$ which will be determined \emph{a posteriori}. Take $S$ to be the ``hat'' function which vanishes on $[0,\pi/2-h]$ and $[\pi/2+h,\pi]$ and such that $\pr_\varth S = 1$ on $[\pi/2-h,\pi/2]$ and $\pr_\varth S = -1$ on $[\pi/2,\pi/2+h]$. Plugging $S$ in~\eqref{eq:minimisation}, we have
  \begin{align}\label{eq:semiclassicalangular}
    \begin{aligned}
    \la_{m|m|}\int_0^\pi|S|^2\sin\varth\d\varth & \leq \int_0^\pi\De_\varth |\pr_\varth S|^2 \sin\varth\d\varth + \int_{0}^\pi G_m(\varth)|S|^2\sin\varth\d\varth \\
                                                 & \leq 2h + \le(G_m(\pi/2) + \sup_{\varth\in(\pi/2-h,\pi/2+h)}\le(G_m(\varth)-G_m(\pi/2)\ri) \ri)\int_0^\pi|S|^2\sin\varth\d\varth.
    \end{aligned}
  \end{align}
  We have
  \begin{align}\label{est:intS2semiclass}
    \begin{aligned}
      \int_0^\pi|S|^2\sin\varth\d\varth & \geq \frac23 \sin\le(\frac{\pi}2-h\ri) h^3,
    \end{aligned}
  \end{align}
  and, by a direct inspection of $G$ from~\eqref{eq:defanglabis},
  \begin{align}\label{est:GthetaG0semiclass}
    \begin{aligned}
      \sup_{(M,a,k)\in K}\sup_{\varth\in(\pi/2-h,\pi/2+h)}\le(G_m(\varth)-G_m(\pi/2)\ri) 
                     & \leq Ch m^2, 
    \end{aligned}
  \end{align}
  for $|m|\geq 2$, provided that $h>0$ is sufficiently small depending on numerical constants, and where $C=C(K)>0$. Plugging~\eqref{est:intS2semiclass} and~\eqref{est:GthetaG0semiclass} in~\eqref{eq:semiclassicalangular}, we have
  \begin{align}\label{eq:laupperproof}
    \begin{aligned}
      \frac{\la_{m|m|}}{m^2} & \leq \frac{3}{m^2h^2\sin\le(\frac{\pi}2-h\ri)} + \le(\frac{G_m(\pi/2)}{m^2} + Ch \ri).
    \end{aligned}
  \end{align}
  Choosing now $h=\frac{1}{\sqrt{|m|}}$ in~\eqref{eq:laupperproof}, we combine~\eqref{est:lowerboundlaproof} and~\eqref{eq:laupperproof} as
  \begin{align}\label{eq:encadrementproof}
    \begin{aligned}
      \Xi^2 \le(\le(1-a\om_+\ri)^2 - \de\ri) \leq & \frac{\la_{m|m|}}{m^2} \leq \frac{3}{m\sin\le(\frac{\pi}2-\frac1{\sqrt{|m|}}\ri)} + \le(\frac{G_m(\pi/2)}{m^2} + \frac{C}{\sqrt{|m|}} \ri).
    \end{aligned}
  \end{align}
  Using that
  \begin{align*}
    \frac{G_m(\pi/2)}{m^2} & = \Xi^2\le(1-a\om_+\ri)^2 + \frac{2+2a^2k^2}{m^2} \xrightarrow{|m|\to+\infty} \Xi^2\le(1-a\om_+\ri)^2,
  \end{align*}
  uniformly for $(M,a,k)\in K$, we deduce from~\eqref{eq:encadrementproof} that there exists $m_0'=m_0'(K)>0$ such that
  \begin{align*}
    \le|\frac{\la_{m|m|}}{m^2} - \Xi^2\le(1-a\om_+\ri)^2 \ri| \leq \Xi^2\de,
  \end{align*}
  for all $|m|\geq m_0'$ and all $(M,a,k)\in K$, and this finishes the proof of the lemma.
\end{proof}

\section{Uniform no-rebound property}\label{sec:norebound}
This section is dedicated to the proof of the following ``no-rebound'' property.
\begin{proposition}\label{prop:concavityargument}
  Let $K$ be a compact subset of strictly rotating ($a>0$) admissible Kerr-adS parameters. There exists $\varep_0>0$ and $m_0\in\mathbb{N}$ depending on $K$ such that for all parameters $(M,a,k)$ in $K$, if $a-kr_+^2 \leq \varep_0k(r_+^2+a^2)$ and $|m|\geq m_0$, the following holds. Let $R$ be a solution of the radial Teukolsky equation~\eqref{eq:radialstatTeuk} with $\la=\la_{m\ell}$ which is regular at the horizon~\eqref{eq:defreghorradial}. For all $r^\star_0\in(-\infty,0)$, if $\wp[R](r^\star_0) = \wp[R]'(r^\star_0) = 0$, then $R=0$ on $(-\infty,0)$ or $\wp[R]''(r^\star_0)<0$.
\end{proposition}

\subsection{Preliminary estimates for the potential $V$}
First, we extract from~\cite{Gra.Hol23} the following lemma for $V_{00}$.
\begin{lemma}\label{lem:posV00}
  We have $V_{00}(r)>0$ for all $r>r_+$, $\lim_{r\to+\infty}V_{00}(r) = 0$, and $\lim_{r\to+\infty}V'_{00}(r)<0$.
\end{lemma}
\begin{proof}
  In~\cite[Section 5.4]{Gra.Hol23}, a (long) direct computation shows that $P(r):=(r^2+a^2)^4V_{00}(r)$ is a polynomial of degree $7$ in $r$, with $\pr_r^jP(r_+)>0$ for all $0\leq j\leq 7$. Hence, $V_{00}$ is positive on $[r_+,+\infty)$, and decays to $0$ when $r\to+\infty$. Moreover $V'_{00}(r) \sim k^2r^2\pr_r\le(\frac1{r}\ri) \pr_r^7P(r_+) \to -k^2\pr^7_rP(r_+)<0$ when $r\to+\infty$ and this finishes the proof of the lemma.  
\end{proof}


We have the following estimates on $V_1$.
\begin{lemma}\label{lem:posV1}
  For all $|m|\geq 2$ and all $a>0$, we have
  \begin{align}\label{est:ineqV1}
    0 & < \frac{V_1(r)}{(\Xi\om_+m)} < \frac{8}{r}\le(k^2r_+^2+2\ri),
  \end{align}
  for all $r> r_+$.
\end{lemma}
\begin{proof}
  Note that since $a\neq0$, $\Xi\om_+m\neq0$. We have
  \begin{align}\label{eq:calculV1}
    \begin{aligned}
      P(r) := 4^{-1}(\Xi\om_+m)^{-1}(r^2+a^2)^2V_1(r) & = \frac12 \pr_r\Delta (r_+^2-r^2) +2\Delta r \\
                                                & = \le(2k^2r_+^2+1+a^2k^2\ri)r^3 -3Mr^2 \\
                                                & \quad + \le((1+a^2k^2)r_+^2 + 2a^2\ri)r -Mr_+^2,
    \end{aligned}
  \end{align}
  which is a third order polynomial in $r$ with strictly positive leading coefficient. From the first line, we have
  \begin{align*}
    P(r_+) & = 0, \\
    \pr_rP(r_+) & = \le(\frac12\pr_r^2\Delta (r_+^2-r^2) +r\pr_r\De + 2\De\ri)\bigg|_{r=r_+} = \pr_r\Delta(r_+)r_+>0, \\
    \pr_r^2P(r_+) & = \le(\frac12\pr_r^3\Delta (r_+^2-r^2) + 3\pr_r\De\ri)\bigg|_{r=r_+} = 3\pr_r\De(r_+) >0,
  \end{align*}
  hence $P$ is strictly positive on $(r_+,+\infty)$ and this proves the first inequality of~\eqref{est:ineqV1}. From the last line of~\eqref{eq:calculV1} using that $ak<1$, one has
  \begin{align*}
    P(r) \leq 2(k^2r_+^2+1)r^3 + 2(r_+^2+a^2)r \leq 2(k^2r_+^2+2)(r^2+a^2)r, 
  \end{align*}
  for all $r>r_+$, and the second inequality of~\eqref{est:ineqV1} follows. 
\end{proof}

\begin{lemma}\label{lem:boundonV0}
  We have
  \begin{align*}
    \Delta -k^2(r-r_+)^2(r+r_+)^2\geq 0,
  \end{align*}
  for all $r\geq r_+$. 
\end{lemma}
\begin{proof}
  It is enough to note that
  \begin{align*}
    \Delta -k^2(r-r_+)^2(r+r_+)^2 & = \le(2k^2r_+^2+1+a^2k^2\ri)r^2 -2Mr +(a^2-k^2r_+^4)
  \end{align*}
  is a polynomial of degree 2 in $r$ with positive leading coefficient, has a root in $r=r_+$, and is locally positive after $r>r_+$. 
\end{proof}

\begin{lemma}\label{lem:V0'}
  We have
  \begin{align}\label{eq:derivV0prooflemma1}
    V_0[\la_{m\ell}]'(r) & = (\la_{m\ell}-2-a^2k^2) w'(r) - (\Xi\om_+m)^22 \frac{\De}{r^2+a^2} \left(\frac{r^2-r_+^2}{r^2+a^2}\right) \frac{2r(r_+^2+a^2)}{(r^2+a^2)^2},
  \end{align}
  where
  \begin{align*}
    w := \frac{\De}{(r^2+a^2)^2},
  \end{align*}
  and we record that
  \begin{align}\label{eq:derivV0prooflemma2}
    \begin{aligned}
    w'(r) & = \frac{\De}{(r^2+a^2)^4}\left(\pr_r\De(r^2+a^2) -4r\Delta\right) \\
    & = - \frac{\De}{(r^2+a^2)^4}\left(\left(2 -2 a^2k^2\right)r^3 -6M r^2 + 2a^2(1-a^2k^2)r  +2Ma^2\right).
    \end{aligned}
  \end{align}
  In particular, we have $V_0'(r) \xrightarrow{r\to+\infty}0$.
\end{lemma}

\subsection{Estimates when $\Re(V) \leq 0$}
We have the following lemma which will be used in the proofs of Proposition~\ref{prop:concavityargument} and Proposition~\ref{prop:envelope} in the next section.
\begin{lemma}\label{lem:cruxfunction}
  Let
  \begin{align*}
    \NN & := \le\{r\in(r_+,+\infty):V_0[\la_{m\ell}](r) + V_{00}(r)\leq 0\ri\}.
  \end{align*}
  Let $K$ be a compact subset of strictly rotating ($a>0$) admissible Kerr-adS parameters. There exists $\varep_0>0$ and $m_0\in\mathbb{N}$ depending on $K$ such that for all $|m|\geq m_0$, $\ell\geq|m|$, if
    \begin{align}\label{eq:epsiloncondition}
      \varep := \frac{a-kr_+^2}{k(r_+^2+a^2)} \leq \varep_0,
    \end{align}
    then, for all $r\in\NN$,
    \begin{align}\label{est:cruxfunction}
      f(r) \leq -\frac{(\Xi \om_+ m)^2}{r} \DM,
    \end{align}
    where $\DM=\DM(K)>0$ and
    \begin{align*}
      f & := 4|V_1|\sqrt{|V_0[\la_{m\ell}]+V_{00}|} +2V_0'.
    \end{align*}
\end{lemma}
\begin{proof}
  Define
  \begin{align}\label{eq:defvarep0start}
     \varep_0 := \min_{(M,a,k)\in K}\le(\frac12 \frac{k(r_+^2+a^2)}{kr_+^2+a^2} \frac{\om_+^2}{k^2} \varept_0\ri), && \text{with} && \varept_0 := \min\le(\frac12, \varept_1, \varept_2\ri), 
  \end{align}
  where
  \begin{align*}
    \varept_1 & := \frac{1}{\le(2k^2r_+^2+1+a^2k^2\ri)}\le((1+a^2k^2) +\frac{18Mk^2}{1-a^2k^2}\ri),\\
    \varept_2 & := \le(\frac{k\DM}{32(k^2r_+^2+2)}\inf_{(3M,+\infty)}\left(\frac{(r^2+a^2)}{\sqrt{\De}}\right)\ri)^2,
  \end{align*}
  with
  \begin{align*}
    \DM & := \inf_{(3M,+\infty)}\left(2r \frac{\De}{r^2+a^2} \left(\frac{r^2-r_+^2}{r^2+a^2}\right) \frac{2r(r_+^2+a^2)}{(r^2+a^2)^2}\right).
  \end{align*}
  Note that with the above definitions, $\varep_0>0$. Assume that $\varep\leq \varep_0$. By the uniform angular eigenvalue limit~\eqref{lim:lam}, there exists\footnote{$m_0$ can be taken uniform for all parameters in $K$ since $\om_+>0$ (because $a>0$) for all $(M,a,k)\in K$, and since the limit in~\eqref{lim:lam} is uniform.} $m_0=m_0(K)>0$ such that for all $|m|\geq m_0$ and all $\ell\geq|m|$,
  \begin{align}\label{eq:mainboundvarept}
    \varept := -\frac{k^2\lat_{m\ell}}{(\Xi\om_+m)^2} & \leq -\frac{k^2}{(\Xi\om_+)^2} \frac{\lat_{m|m|}}{m^2} \leq \frac{k^2}{(\Xi\om_+)^2} \Xi^2 \varep \frac{kr_+^2+a}{k(r_+^2+a^2)} + \frac{\varept_0}{2} \leq \varept_0,
  \end{align}
  for all parameters $(M,a,k)\in K$. Now we assume that $\NN\neq\emptyset$ (if else there is nothing to prove). Using that $V_{00}>0$ (Lemma~\ref{lem:posV00}) and Lemma~\ref{lem:boundonV0}, this implies that $\varept>0$. Furthermore, using again Lemmas~\ref{lem:posV00} and~\ref{lem:boundonV0}, using the bounds on $V_1$ (Lemma~\ref{lem:posV1}), and finally using that $\varept \leq \varept_2$ by~\eqref{eq:mainboundvarept}, one has
  \begin{align}\label{eq:prelimfproofunif}
    \begin{aligned}
    f(r) & = 4|V_1(r)|\sqrt{-V_0[\la_{m\ell}](r)-V_{00}(r)} +2V_0[\la_{m\ell}]'(r) \\
         & \leq 32\frac{(\Xi \om_+m)^2}{kr}(k^2r_+^2+2)\sqrt{\varept}\frac{\sqrt{\De}}{(r^2+a^2)} + 2V_0[\la_{m\ell}]'(r) \\
         & \leq \frac{(\Xi\om_+m)^2}{r}D + 2V_0[\la_{m\ell}]'(r),
       \end{aligned}
  \end{align}
   for all $r\in\NN$, where $w := \frac{\De}{(r^2+a^2)^2}$. Our goal now is to obtain a bound on $V'_0$.\\

  First we show that $\NN \subset \le[3M/\Xi,+\infty\ri)$ using~\eqref{eq:mainboundvarept}. From the definition of $V_0$, one has
  \begin{align*}
    \left(\frac{k}{\Xi\om_+m}\right)^2(r^2+a^2)^2[V_0[\la_{m\ell}](r) & = (1-\varept)\Delta -k^2(r-r_+)^2(r+r_+)^2 \\
                                                                     & = \le(2k^2r_+^2+1+a^2k^2\ri)r^2 -2Mr +(a^2-k^2r_+^4) \\
                                                                     & \quad -\varept(r^2+a^2)(1+k^2r^2) + 2Mr\varept\\
                                                                     & =: Z(r),
  \end{align*}
  and we note that 
  \begin{align}\label{eq:d2Zconvex}
    \pr_r^2Z(r) & = 2 \le(2k^2r_+^2+1+a^2k^2\ri) -2\varept\le(1+a^2k^2\ri) - 12\varept k^2r^2.
  \end{align}
  By~\eqref{eq:mainboundvarept} (since $\varept \leq \varept_1$), one has
  \begin{align}\label{eq:conditionvarep2}
     \frac{3M}{1-a^2k^2} \leq \frac16 \varept^{-1}k^{-2}\le(2k^2r_+^2+1+a^2k^2\ri) - \frac16 k^{-2}(1+a^2k^2) =: r_{\varept}.
  \end{align}
  The fourth order polynomial $Z$ has negative leading coefficient $-k^2\varept$ (since $\varept>0$), has a root at $r=r_+$, is positive locally after $r>r_+$ if $\varept<1$ (which is true by~\eqref{eq:mainboundvarept} since $\varept\leq 1/2$), and is convex for\footnote{Using~\eqref{eq:conditionvarep2}, one has $r_\varept\geq 3M >r_+$.} $r_+< r < r_\varept$ by~\eqref{eq:d2Zconvex}. Hence $V_0(r)>0$ for all $r\in(r_+,r_\varept)$, which, using~\eqref{eq:conditionvarep2}, implies
  \begin{align}\label{eq:posV0vareptbis}
    \NN \subset \le[3M/\Xi,+\infty\ri),
  \end{align}
  as desired.\\
  
  We can now estimate $V_0[\la_{m\ell}]'$ using the computations of Lemma~\ref{lem:V0'}. From~\eqref{lim:lam}, for $|m|$ sufficiently large depending on $K$, one has $\la_{m\ell}-2-a^2k^2>0$ for all $\ell\geq |m|$. Combined with~\eqref{eq:derivV0prooflemma1} and~\eqref{eq:derivV0prooflemma2}, using~\eqref{eq:posV0vareptbis}, we infer
  \begin{align}\label{eq:V'0leqDM}
    \begin{aligned}
      r\in \NN \implies r \geq \frac{3M}{1-a^2k^2} & \implies \left(2 -2 a^2k^2\right)r^3 -6M r^2 \geq 0 \\
      & \implies w'(r)<0 \\
      & \implies V_0[\la_{m\ell}]'(r) < -\frac{(\Xi \om_+ m)^2}{r} \DM,
    \end{aligned}
  \end{align}
  where $\DM>0$ is the constant from the definition~\eqref{eq:defvarep0start} of $\varep_0$. Plugging~\eqref{eq:V'0leqDM} in~\eqref{eq:prelimfproofunif} we obtain the desired~\eqref{est:cruxfunction} and this finishes the proof of the lemma.
\end{proof}

\subsection{Proof of Proposition~\ref{prop:concavityargument}}
Using the regularity at the horizon~\eqref{eq:defreghorradial}, we deduce from~\eqref{eq:wronskien} and Lemma~\ref{lem:posV1} that
\begin{align}\label{eq:posW}
  W(r^\star) & = 2i\int_{-\infty}^{r^\star} V_1|R|^2\,\d r^\star =
               \begin{cases}
                 +i |W(r^\star)| & \text{if $m>0$} \\
                 -i |W(r^\star)| & \text{if $m<0$},
               \end{cases}
\end{align}
for all $r^\star\in(-\infty,0)$. Now, assume that at $r^\star=r^\star_0$, we have $\wp(r^\star_0) = \wp'(r^\star_0) = 0$. If $\Re(V(r^\star_0))>0$, we immediately conclude from~\eqref{eq:wpderiv1} that $R(r^\star_0)=R'(r^\star_0)=0$ hence $R=0$ on $(-\infty,0)$ by Cauchy-Lipschitz. If else, we have $|R'|=\sqrt{-\Re(V)}|R|$ at $r^\star=r^\star_0$, and one can assume that $|R|\neq0$ otherwise $R=0$ on $(-\infty,0)$ by Cauchy-Lipschitz. Moreover, from the vanishing of $\wp(r^\star_0)$ and~\eqref{eq:posW}, this yields
\begin{align*}
  2R\bar R' = W = \pm i |W| = \pm2i |R||R'| = \pm2i \sqrt{-\Re(V)}|R|^2 
\end{align*}
at $r^\star=r^\star_0$, with signs $+$ if $m>0$ and $-$ if $m<0$. Plugging this in the expression~\eqref{eq:wpderiv2} of $\wp''$, we deduce that
\begin{align*}
  \wp''(r^\star_0) & =  \left(f(r^\star_0) + 2V_{00}'(r^\star_0)\right) |R|^2,
\end{align*}
where $f$ is the function defined in Lemma~\ref{lem:cruxfunction} (recall that $V_1>0$ if $m>0$ and $V_1<0$ if $m<0$ by Lemma~\ref{lem:posV1}). From~\eqref{est:cruxfunction}, we have
\begin{align*}
  f(r_0) + 2V_{00}'(r_0) & \leq - \frac{(\Xi\om_+m)^2}{r_0}\DM + 2V'_{00}(r_0). 
\end{align*}
Using that $\lim_{r\to+\infty}V_{00}'(r)<0$ (Lemma~\ref{lem:posV00}), the map $r\mapsto - \frac{(\Xi\om_+m)^2}{r}\DM + 2V'_{00}(r)$ is uniformly negative on $[r_+,+\infty]$ provided that $|m|$ is sufficiently large depending on $K$, and we deduce that $\wp''(r^\star_0)<0$ as desired.

\section{High-frequency solutions above the Hawking-Reall bound}\label{sec:envelope}
This section is dedicated to the proof of the following proposition.
\begin{proposition}\label{prop:envelope}
  Let $K$ be a compact subset of strictly rotating ($a>0$) admissible Kerr-adS parameters. There exists $\varep_0>0$ depending on $K$ such that for all parameters $(M,a,k)$ in $K$ with $0 < a-kr_+^2 \leq \varep_0k(r_+^2+a^2)$, there exists $m_1(M,a,k)>0$ such that for all $|m|\geq m_1$ the following holds. For all solutions $R$ to the radial Teukolsky equation~\eqref{eq:radialstatTeuk} with $\la=\la_{m|m|}$, there exists $r^\star\in(-\infty,0)$ such that $\wp[R](r^\star)<0$. 
\end{proposition}
\begin{remark}
  The constant $\varep_0$ in Proposition~\ref{prop:envelope} is the same as the one in Proposition~\ref{prop:concavityargument}.
\end{remark}

In this section, we shall assume that $(M,a,k)$ are fixed admissible Kerr-adS parameters above but close to the Hawking-Reall bound:
\begin{align}\label{eq:conditionepsilonprecioussec}
  0 < a-kr_+^2 \leq \varep_0k(r_+^2+a^2),
\end{align}
where $\varep_0$ is the constant from Lemma~\ref{lem:cruxfunction}. We will consider the radial Teukolsky equation with $\la=\la_{m|m|}$. In view of~\eqref{lim:lam}, since $a>kr_+^2$, we will assume in the rest of this section that $|m|$ is sufficiently large depending on $(M,a,k)$ so that $\lat_{m|m|} <0$ and we will define
\begin{align*}
  \omo & := k(-\lat_{m|m|})^{1/2}. 
\end{align*}

\subsection{The WKB approximation}\label{sec:WKB}
Let $(R_{1},R_2)$ be the basis of solutions to the radial Teukolsky ODE~\eqref{eq:radialstatTeuk} determined by the boundary conditions
\begin{align*}
  R_1 & = R_2 = \omo^{-1/2}, & R'_1 & = i\omo R_1, & R'_2 & = -i\omo R_2,
\end{align*}
at $r^\star=0$. From~\cite[Theorem 2]{Olv61}, we have the following lemma.
\begin{lemma}\label{lem:WKB}
  There exists two smooth functions $\varep_1,\varep_2:(-\infty,0]\to\mathbb{C}$ such that
  \begin{align}\label{eq:WKBsolutions}
    \begin{aligned}
      R_1 (r^\star) & = \omo^{-1/2} \le(e^{i \omo r^\star}+\varep_1(r^\star)\ri), & R_2 (r^\star) & = \omo^{-1/2} \le(e^{-i\omo r^\star}+\varep_2(r^\star)\ri),
    \end{aligned}
  \end{align}
  for all $r^\star\in(-\infty,0]$, and the functions $\varep_i$ satisfy
  \begin{align}\label{est:WKBerrorbound}
    |\varep_i(r^\star)|,~\omo^{-1}|\varep_i'(r^\star)| & \leq \exp(F(r^\star))-1, 
  \end{align}
  for all $r^\star\in (-\infty,0]$, where
  \begin{align*}
    F(r^\star) & := \frac1\omo\int^0_{r^\star} \left|V[\la_{m|m|}](r^{\star,'}) + \omo^2\right|\d r^{\star,'}.
  \end{align*}
\end{lemma}

The following lemma gives a bound on the error terms in Lemma~\ref{lem:WKB}.
\begin{lemma}\label{lem:bounderrorcontrolfunction}
  There exists a constant $C=C(M,a,k)>0$ and $m_0=m_0(M,a,k)>0$ such that 
  \begin{align}\label{est:omhigh}
    C^{-1}|m| \leq \omo \leq C |m|,
  \end{align}
  for all $|m|\geq m_0$, and
  \begin{align}\label{est:bounderrorcontrolfunction}
     |F(r^\star)| \les_{M,a,k} |m||r^\star|^2,
  \end{align}
  for all $r^\star\in[-\pi/\omo,0]$ and all $|m|\geq m_0$.  
\end{lemma}
\begin{proof}
  The first bound~\eqref{est:omhigh} follows from~\eqref{lim:lam}. Since we precisely chose $\omo$ so that $\omo^2=-V_0[\la_{m|m|}](0)$, and since $V_0$ is a smooth function of $r^\star$, Taylor-Lagrange inequality gives
  \begin{align*}
    \le|V_0[\la_{m|m|}](r^\star)+\omo^2\ri| & \leq \le(\sup_{r^\star\in[-1,0]} |V_0[\la_{m|m|}]'(r^\star)|\ri) |r^\star| \les_{M,a,k} m^2|r^\star|,  
  \end{align*}
  for all $r^\star\in[-1,0]$. Similarly, using that $V_{00}(0)=V_1(0)=0$ (Lemmas~\ref{lem:posV00} and~\ref{lem:posV1}), one has
  \begin{align*}
    |V_{00}| & \les_{M,a,k} |r^\star|, & |V_1| & \les_{M,a,k} |m| |r^\star|,
  \end{align*}
  for all $r^\star\in[-1,0]$. We infer that
  \begin{align*}
    |V[\la_{m|m|}]+\omo^2| & \leq \le|V_0[\la_{m|m|}]+\omo^2\ri| + |V_{00}| + |V_1| \les_{M,a,k} m^2|r^\star| + |r^\star| + |m| |r^\star| \les_{M,a,k} m^2|r^\star|,
  \end{align*}
  for all $r^\star\in[-1,0]$ and~\eqref{est:bounderrorcontrolfunction} follows by integration and~\eqref{est:omhigh}.
\end{proof}

\subsection{Proof of Proposition~\ref{prop:envelope}: the case without interference}\label{sec:nointerferences}
If $R$ is colinear to $R_1$ or $R_2$, Proposition~\ref{prop:envelope} will directly follow from the following (stronger) lemma.
\begin{lemma}\label{lem:nointerf}
  We have $\wp[R_i](r^\star) < 0$ for all $r^\star\in[-\pi/\omo,0]$, provided that $|m|$ is sufficiently large depending on $(M,a,k)$.
\end{lemma}
\begin{proof}
  We write $\wp_i := \wp[R_i]$ and $W_i := W[R_i]$. We rewrite equation~\eqref{eq:wpderiv2} as
  \begin{align}\label{eq:oscillateurwp}
    \wp_i''+ (2\omo)^2\wp_i & = -2\Im(V[\la_{m|m|}])\Im(W_i) + 2\Re(V[\la_{m|m|}]')|R_i|^2 + 4\le(\Re(V[\la_{m|m|}])+\omo^2\ri)\wp_i =: E_i,
  \end{align}
  Duhamel's formula, using that $\wp_i=\wp'_i=0$ at $r^\star=0$ (see~\eqref{eq:WKBsolutions}), gives
  \begin{align}\label{eq:DuhamelwpR2}
    \wp_i(r^\star) & = \frac{1}{2\omo}\int^{r^\star}_0E_i(r^{\star,'})\sin\le(2\omo\le(r^\star-r^{\star,'}\ri)\ri)\d r^{\star,'}.
  \end{align}
  We rewrite the source term $E_i$ of~\eqref{eq:oscillateurwp} as $E_i=E_i^s+E_i^e$ with
  \begin{align*}
    E_i^s & := 4\sigma_iV_1 + 2\omo^{-1}\Re(V[\la_{m|m|}]'), \\
    E_i^e & :=  2V_1\le(\Im(W_i)-2\sigma_i\ri) + 2\Re(V[\la_{m|m|}]')\le(|R_i|^2-\omo^{-1}\ri) + 4\le(\Re(V[\la_{m|m|}])+\omo^2\ri)\wp_i, 
  \end{align*}
  with $\sigma_i=(-1)^i$. From the WKB estimates~\eqref{eq:WKBsolutions},~\eqref{est:WKBerrorbound}, \eqref{est:omhigh}, \eqref{est:bounderrorcontrolfunction}, we have
  \begin{align}\label{est:WKBappl1}
    \le|\Im(W_i)-2\sigma_i\ri|, |\wp_i| & \les_{M,a,k} |m||r^\star|^2, & \le||R_i|^2-\omo^{-1}\ri| & \les_{M,a,k} |r^\star|^2,
  \end{align}
  for all $r^\star\in[-\pi/\omo,0]$. By Taylor-Lagrange, using that $V_1(0) = V'_0(0) = V_{00}(0) = 0$, (Lemmas~\ref{lem:posV00}, \ref{lem:posV1} and~\ref{lem:V0'}), we have
  \begin{align}\label{est:potentialsnointerf}
    \begin{aligned}
      |V_1| & \les_{M,a,k} |m| |r^\star|, \\
      |2\Re(V[\la_{m|m|}]')| & \les_{M,a,k} m^2 |r^\star| +1 \les_{M,a,k} |m|, \\
      \le|\Re(V[\la_{m|m|}])+\omo^2\ri| & \les_{M,a,k} m^2|r^\star|^2+|r^\star| \les_{M,a,k} |m||r^\star|,
    \end{aligned}
  \end{align}
  for all $r^\star\in[-\pi/\omo,0]$. Combining~\eqref{est:WKBappl1} and~\eqref{est:potentialsnointerf}, we get $|E^e_i| \les |m||r^\star|^2$ for all $r^\star\in[-\pi/\omo,0]$. By Taylor-Lagrange, using that $V_1(0) = V_0[\la_{m|m|}]'(0) = 0$ and using~\eqref{est:omhigh}, one has
  \begin{align*}
    \le|E_i^s  - 4\si_iV_1'(0)r^\star - 2\omo^{-1}V'_{00}(0) - 2 \omo^{-1}\le(V_0[\la_{m|m|}]''(0) +V_{00}''(0)\ri)r^\star \ri| & \les_{M,a,k} m |r^\star|^2,
  \end{align*}
  which yields
  \begin{align}\label{eq:Eibound}
    \le|E_i(r^\star) - \al - (\be_i +\ga)r^\star \ri| & \les_{M,a,k} |m||r^\star|^2,
  \end{align}
  for all $r^\star\in[-\pi/\omo,0]$, with
  \begin{align*}
    \al & := 2 \omo^{-1}V'_{00}(0), & \be_i & := \le(4\sigma_iV_1 +2 \omo^{-1}V_0[\la_{m|m|}]'\ri)'(0), & \ga & :=  2 \omo^{-1}V''_{00}(0).
  \end{align*}
  Thus, plugging~\eqref{eq:Eibound} in~\eqref{eq:DuhamelwpR2}, we get
  \begin{align}\label{est:Taylorwp1}
    \le|\wp_i(r^\star)- \frac{\al}{(2\omo)^2} \le(1-\cos(2\omo r^\star)\ri) - \frac{(\be_i+\ga)}{(2\omo)^3}\le(2\omo r^\star -\sin(2\omo r^\star)\ri)\ri| & \leq C |r^\star|^3,
  \end{align}
  for all $r^\star\in[-\pi/\omo,0]$ and where $C=C(M,a,k)>0$. In the two cases $i=1,2$, we have
  \begin{align*}
    \be_i & = -k^2 \lim_{r\to+\infty} r \le(4\sigma_iV_1 +2 \omo^{-1}V_0[\la_{m|m|}]'\ri) \\
          & \geq -k^2 \lim_{r\to+\infty} r \le(4|V_1| +2 \omo^{-1}V_0[\la_{m|m|}]'\ri) \\
          & = - k^2\omo^{-1}\lim_{r\to+\infty}rf(r).
  \end{align*}
  where $f$ is the function defined in Lemma~\ref{lem:cruxfunction}. Since $V_0[\la_{m|m|}](0)+V_{00}(0) = -\omo^2 < 0$, there exists $r_0>r_+$ such that $[r_0,+\infty)\subset \NN$ where $\NN$ is the set defined in Lemma~\ref{lem:cruxfunction}. Hence, from~\eqref{est:cruxfunction} -- using that $\varep \leq \varep_0$ where $\varep_0$ is the constant in Lemma~\ref{lem:cruxfunction} and that one can take $|m|\geq m_0$ with $m_0$ the constant in Lemma~\ref{lem:cruxfunction} --, we infer that 
  \begin{align}\label{eq:beiesti}
    \be_i \geq k^2\omo^{-1}\DM(\Xi\om_+m)^2 \geq B\omo,
  \end{align}
  where $B=B(M,a,k)>0$. We have
  \begin{align}
    \label{eq:sinxx3}
     \forall x \in[-2\pi,0], \quad 4\pi^2\le(x-\sin x\ri) \leq x^3 \leq 0. 
  \end{align}
  Using~\eqref{eq:beiesti} (and in particular that $\be_i>0$) and~\eqref{eq:sinxx3} we deduce from~\eqref{est:Taylorwp1} that for all $r^\star\in[-\pi/\omo,0]$,
  \begin{align*}
    \wp_i(r^\star) & \leq \frac{2V_{00}'(0)}{4\omo^3} \le(1-\cos(\omo r^\star)\ri) + \frac{2V_{00}''(0)}{8\omo^4} \le(2\omo r^\star -\sin(2\omo r^\star)\ri) \\
                 & \quad + \frac{B}{8\omo^2} \le(2\omo r^\star -\sin(2\omo r^\star)\ri) + \frac{C}{\omo^3}|\omo r^\star|^3 \\
                 & \leq  \frac{2V_{00}'(0)}{4\omo^3} \le(1-\cos(\omo r^\star)\ri) + \frac{B}{4\pi^2\omo^2} (\omo r^\star)^3 + \frac{A}{\omo^3}|\omo r^\star|^3,
  \end{align*}
  where $A=A(M,a,k)>0$. Hence, using that $V'_{00}(0)\leq 0$ (Lemma~\ref{lem:posV00}), that $\omo r^\star \leq 0$, and provided that $\omo$ is sufficiently large depending on $(M,a,k)$, we have $\wp_i(r^\star)\leq0$ for all $r^\star\in[-\pi/\omo,0]$ for $i=1,2$, as desired.
\end{proof}

\subsection{Proof of Proposition~\ref{prop:envelope}: the case with interference}\label{sec:proofenvelope}
Let $R$ be a solution to the radial Teukolsky equation~\eqref{eq:radialstatTeuk}. Let $A_1,A_2 \in\mathbb{C}$ such that $R = A_1R_1+A_2R_2$. Assume that $A_1A_2\neq0$ (the other case is treated in Section~\ref{sec:nointerferences}). Using Lemma~\ref{lem:nointerf}, we have
\begin{align}\label{eq:wpinterf}
  \begin{aligned}
    \wp[R] & = |A_1|^2\wp[R_1] + |A_2|^2\wp[R_2] +2\Re\le(A_1\bar A_2 R_1\bar R_2' + A_2\bar A_1R_2\bar R_1'\ri) \\
           & \leq 2\Re\le(A_1\bar A_2 R_1\bar R_2' + A_2\bar A_1R_2\bar R_1'\ri) =: q[R],
  \end{aligned}
\end{align}
on $[-\pi/\omo,0]$. From the WKB estimates~\eqref{eq:WKBsolutions},~\eqref{est:WKBerrorbound}, \eqref{est:omhigh}, \eqref{est:bounderrorcontrolfunction}, we have
\begin{align*}
  \le|\omo R_1\bar R_2 - e^{2i\omo r^\star}\ri| + \omo^{-1/2}\le|R_1' -i \omo R_1\ri| + \omo^{-1/2}\le|R_2' +i \omo R_2\ri| & \les_{M,a,k} |m||r^\star|^2,
\end{align*}
hence
\begin{align}\label{eq:approxqq}
  \le|q(r^\star) + 4\Im\le(A_1\bar A_2e^{2i\omo r^\star}\ri)\ri|\les_{M,a,k} |A_1||A_2| |m||r^\star|^2,
\end{align}
for all $r^\star\in[-\pi/\omo,0]$. Up to renormalising $R$ one can assume that $|A_1||A_2|=1$, \emph{i.e.} $A_1\bar A_2 = e^{i\theta}$ with $\theta\in[0,2\pi[$. With this normalisation, one has
\begin{align}\label{eq:qafterrenorm}
  \le|q(r^\star)+ 4\sin\le(2\omo r^\star+\theta\ri)\ri| & \leq C |m||r^\star|^2,
\end{align}
for all $r^\star\in[-\pi/\omo,0]$ and where $C=C(M,a,k)>0$. For all $\theta\in[0,2\pi[$, there exists $\kappa_\theta\in\{-3,1\}$ such that $-\frac{\theta}{2\omo} + \kappa_\theta\frac\pi{4\omo}\in [-\pi/\omo,0]$. Taking $r^\star = -\frac{\theta}{2\omo} + \kappa_\theta\frac\pi{4\omo}$ in~\eqref{eq:qafterrenorm}, we get
\begin{align*}
  q\left(-\frac{\theta}{2\omo} + \kappa_\theta\frac\pi{4\omo}\right) & \leq -4 + \frac{C |m|}{\omo^2}\le(-\frac\theta2 + \kappa_\theta\frac{\pi}4\ri)^2 <0,
\end{align*}
provided that $|m|$ is sufficiently large depending on $(M,a,k)$. Using~\eqref{eq:wpinterf} this finishes the proof of Proposition~\ref{prop:envelope}.

\section{Proofs of the main theorems}\label{sec:conclusion}
\subsection{The mode stability result: proof of Theorem~\ref{thm:modestab}}\label{sec:proofthmmodestab}
The proof follows the arguments of~\cite[Section 5]{Gra.Hol23}. Let $(M,a,k)$ be fixed admissible Kerr-adS parameters such that $a<kr_+^2$. Assume that there exists non-trivial stationary modes~\eqref{eq:defstatmodes} solutions to the Teukolsky equations~\eqref{eq:Teuk}. Up to decomposing $S^{[\pm2]}(\varth)e^{\pm im\varphi}$ onto the Hilbert basis of Lemma~\ref{lem:hilbert} (see also Remark~\ref{rem:spin-2}), one can assume that $R:=R^{[+2]}$ is solution to the radial Teukolsky equation~\eqref{eq:radialstatTeuk} with $\la=\la_{m\ell}$, $\ell\geq |m|$. The regularity condition at the horizon~\eqref{eq:defreghor} and the boundary condition at infinity~\eqref{eq:TeukBC} imply that $R$ is regular at the horizon~\eqref{eq:defreghorradial} and satisfies the condition $\wp[R](0)=0$ (Lemma~\ref{lem:sep} and Lemma~\ref{lem:reciproque}).\\

Integrating~\eqref{eq:wpderiv1} and using that $\wp[R](r^\star)\to0$ when $r^\star\to-\infty$ by~\eqref{eq:defreghorradial}, one has the energy identity
\begin{align}\label{eq:energy}
  \wp[R](r^\star) & = 2 \int_{-\infty}^{r^\star} \left(\le|R'\ri|^2 +\left(V_0[\la_{m\ell}]+V_{00}\right)|R|^2\right)\,\d r^{\star,'},
\end{align}
for all $r^\star\in(-\infty,0]$. We have $V_{00}(r)>0$ (see Lemma~\ref{lem:posV00}) and
\begin{align*}
  V_0[\la_{m\ell}](r) & = \frac{\De}{(r^2+a^2)^2}\lat_{m\ell} + \le(\frac{\Xi\om_+m}{k}\ri)^2\frac{\De-k^2(r-r_+)^2(r+r_+)^2}{(r^2+a^2)^2} \geq \frac{\De}{(r^2+a^2)^2}\lat_{m\ell},
\end{align*}
for all $r\in(r_+,+\infty)$ (see Lemma~\ref{lem:boundonV0}). By Lemma~\ref{lem:hilbert} and Lemma~\ref{lem:angular}, one has
\begin{align*}
  \lat_{m\ell} & \geq \la_{m|m|} -\Xi^2 \le(\frac{\om_+}{k}\ri)^2m^2 -2-a^2k^2 \sim\Xi^2\frac{(kr_+^2-a)(kr_+^2 + a)}{k^2(r_+^2+a^2)^2}m^2 \xrightarrow{|m|\to+\infty} +\infty,
\end{align*}
since the (strict) Hawking-Reall bound $a<kr_+^2$ holds.\footnote{Note that this limit is not uniform in all parameters satisfying the Hawking-Reall bound.} Hence, there exists $m_2=m_2(M,a,k)>0$ sufficiently large such that for all $|m|\geq m_2$, the function $V_0[\la_{m\ell}]+V_{00}$ is strictly positive and, by~\eqref{eq:energy}, $\wp[R](0)=0$ implies $R=R'=0$ and this finishes the proof of Theorem~\ref{thm:modestab}.

\subsection{The shooting argument: proof of Theorem~\ref{thm:main}}\label{sec:proofthmmain}
We consider $\Theta:s\in[0,1]\mapsto\left(M(s),a(s),k(s)\right)$ a continuous path of admissible Kerr-adS parameters, such that
\begin{align*}
  a(0) & < k(0)(r_+(0))^2 \quad \text{and} \quad a(1) > k(1)(r_+(1))^2.
\end{align*}
Upon restricting the path $\Theta$, one can assume that all the black holes are strictly rotating, \emph{i.e.} $a(s) >0$ for all $s\in[0,1]$, so that Propositions~\ref{prop:concavityargument} and~\ref{prop:envelope} apply. Let $K=\Theta([0,1])$. By continuity, there exists $s_c\in]0,1]$ such that
\begin{align}\label{eq:condvarepproofmain}
  a(s_c) > k(s_c)(r_+(s_c))^2 && \text{and} && \forall s\in[0,s_c], \quad \varep(s) := \frac{a(s)-k(s)r_+(s)^2}{k(s)(r_+(s)^2+a(s)^2)} \leq \varep_0,
\end{align}
where $\varep_0=\varep_0(K)>0$ is the constant of Propositions~\ref{prop:concavityargument} and~\ref{prop:envelope}.\\

Let $m_0=m_0(K)>0$ be the constant of Proposition~\ref{prop:concavityargument}, $m_1=m_1\left(M(s_c),a(s_c),k(s_c)\right)>0$ be the constant of Proposition~\ref{prop:envelope} and $m_2=m_2\left(M(0),a(0),k(0)\right)>0$ be the constant of the proof of Theorem~\ref{thm:modestab}, and fix $|m| \geq \max(m_0,m_1,m_2)$.\\ 

Let $R_s$ be the regular solution at the horizon~\eqref{eq:defreghorradial} of the radial Teukolsky equation~\eqref{eq:radialstatTeuk} for the parameters $\left(M(s),a(s),k(s)\right)$, with constant $\la=\la_{m|m|}(s)$, normalised by
\begin{align}\label{eq:normalisationL2}
  \int_{-\infty}^{0} |R_s|^2 \,\d r^\star = 1. 
\end{align}
We state the following lemma, which will be used in the sequel. The proof is standard and is left to the reader.
\begin{lemma}\label{lem:continuityproofmain}
  The map $s\in[0,1]\mapsto \wp[R_s] \in C^0_0\left((-\infty,0]_{r^\star},\mathbb{R}\right)$ is continuous.\footnote{$C^0_0\left((-\infty,0]_{r^\star},\mathbb{R}\right)$ denotes the space of continuous functions vanishing at $-\infty$ endowed with the uniform norm.} 
\end{lemma}

Using Lemma~\ref{lem:boundonV0}, there exists $r_c>r_+$ sufficiently small with respect to $K$ and $|m|$ such that
\begin{align*}
  V_0[\la_{m|m|}](r) + V_{00}(r) & \geq  \frac{\De}{(r^2+a^2)^2}\le(\min_{s\in[0,1]}\lat_{m|m|}\ri) + V_{00}(r) \\
                     & > \frac12 V_{00}(r_+) = \frac12\frac{(\pr_r\De(r_+))^2}{(r_+^2+a^2)^2} >0,
\end{align*}
for all $r\in (r_+,r_c)$ and all $s\in[0,1]$. Hence, by the energy identity~\eqref{eq:energy}, we have
\begin{align}
  \label{eq:QRs>0rc}
  \forall s\in[0,1],~\forall r\in(r_+,r_c], \quad \wp[R_s](r)>0.
\end{align}

Let us now define
\begin{align*}
  s_m := \sup J, && \text{with} && J := \le\{s\in[0,s_c]:\forall s'\in[0,s],~\forall r^\star\in(r^\star_c,0),~\wp[R_{s'}](r^\star)\geq0\ri\}.
\end{align*}
We have the following observations.
\begin{enumerate}
\item Since $|m|\geq m_2$ where $m_2$ is the constant of the proof of Theorem~\ref{thm:modestab} associated to $\left(M(0),a(0),k(0)\right)$, the function $V_0[\la_{m|m|}]+V_{00}$ is strictly positive on $(-\infty,0)$ for $s=0$ and by the energy identity~\eqref{eq:energy}, $\wp[R_0]>0$ on $[r^\star_c,0]$. Hence $0\in J$ and $s_m$ is well-defined.
\item By continuity (Lemma~\ref{lem:continuityproofmain}), we have $\wp[R_{s_m}](r^\star) \geq 0$, for all $r^\star\in[r^\star_c,0]$.
\item By~\eqref{eq:condvarepproofmain} and since $|m|\geq m_1$ where $m_1$ the constant of Proposition~\ref{prop:envelope} associated to $\left(M(s_c),a(s_c),k(s_c)\right)$, Proposition~\ref{prop:envelope} implies that $s_c\notin J$. By continuity (Lemma~\ref{lem:continuityproofmain}), this implies that there exists $r^\star_0\in[r^\star_c,0]$ such that $\wp[R_{s_m}](r^\star_0) = 0$.
\end{enumerate}
We can now conclude. Assume by contradiction that $r^\star_0\neq0$. By~\eqref{eq:QRs>0rc} one has $r^\star_c<r^\star_0<0$. Hence, since $\wp[R_{s_m}]$ is non-negative, one must have $\wp[R_{s_m}]'(r^\star_0)=0$. By the result of Proposition~\ref{prop:concavityargument} (which applies, since $|m|\geq m_0$ and~\eqref{eq:condvarepproofmain} holds), either $R_{s_m}=0$ which is impossible by~\eqref{eq:normalisationL2}, or $\wp[R_{s_m}]''(r^\star_0)<0$ which is also impossible because $\wp[R_{s_m}]$ is non-negative. Hence $r^\star_0=0$, \emph{i.e.} $R_{s_m}$ satisfies the target boundary condition~\eqref{eq:BCstat}. Thus, by Lemma~\ref{lem:sep} and Lemma~\ref{lem:target}, $R_{s_m}$ defines a non-trivial stationary mode solution, as claimed in Theorem~\ref{thm:main}.\\

From the mode stability result of Theorem~\ref{thm:modestab} one has
\begin{align*}
  \liminf_{|m|\to+\infty} \left(a(s_m)-k(s_m)r_+(s_m)^2\right) \geq 0.
\end{align*}
On the other hand, there exists $s_{\mathrm{HR}}\in[0,s_c)$ such that
\begin{align*}
  a(s_{\mathrm{HR}}) & = k(s_{\mathrm{HR}}) \le(r_+(s_{\mathrm{HR}})\ri)^2, & \forall s \in(s_{\mathrm{HR}},s_c], & \quad 0 < \varep(s) \leq \varep_0.
\end{align*}
Hence, from the high-frequency result of Proposition~\ref{prop:envelope}, for all $s\in(s_{\mathrm{HR}},s_c)$, there exists $m_3=m_3(s)>0$ such that for all $|m|\geq m_3$, $s\notin J=[0,s_m]$. Thus, 
\begin{align*}
  \limsup_{|m|\to+\infty} \left(a(s_m)-k(s_m)r_+(s_m)^2\right) \leq 0.
\end{align*}
This shows~\eqref{eq:limHRminfinity} and finishes the proof of Theorem~\ref{thm:main}.

\bibliographystyle{graf_GR_alpha}
\bibliography{graf_GR}

\newcommand{\etalchar}[1]{$^{#1}$}
\begin{thebibliography}{CDH{\etalchar{+}}14}

\bibitem[CD04]{Car.Dia04}
V.~Cardoso, {\'O}.~J.~C. Dias, \emph{Small {{Kerr}}--anti-de {{Sitter}} black
  holes are unstable}, Phys. Rev. D 70 (2004), no.~8, 084011.

\bibitem[CDH{\etalchar{+}}14]{Car.Dia.Har.Leh.San14}
V.~Cardoso, {\'O}.~J.~C. Dias, G.~S. Hartnett, L.~Lehner, J.~E. Santos,
  \emph{Holographic thermalization, quasinormal modes and superradiance in
  {{Kerr-AdS}}}, J. High Energ. Phys. 2014 (2014), no.~4, 183.

\bibitem[CDLY04]{Car.Dia.Lem.Yos04}
V.~Cardoso, {\'O}.~J.~C. Dias, J.~P.~S. Lemos, S.~Yoshida, \emph{Black-hole
  bomb and superradiant instabilities}, Phys. Rev. D 70 (2004), no.~4, 044039.

\bibitem[CDY06]{Car.Dia.Yos06}
V.~Cardoso, {\'O}.~J.~C. Dias, S.~Yoshida, \emph{Classical instability of
  {{Kerr-AdS}} black holes and the issue of final state}, Phys. Rev. D 74
  (2006), no.~4, 044008.

\bibitem[Col21]{Col21}
S.~C. Collingbourne, \emph{The {{Gregory}}--{{Laflamme}} instability of the
  {{Schwarzschild}} black string exterior}, J. Math. Phys. 62 (2021), no.~3,
  032502.

\bibitem[Col23]{Col23}
S.~C. Collingbourne, \emph{Coercivity properties of the canonical energy in
  double null gauge on the 4-dimensional {{Schwarzschild}} exterior}, Class.
  Quantum Grav. 40 (2023), no.~22, 225013.

\bibitem[CS17]{Cho.Shl17}
O.~Chodosh, Y.~{Shlapentokh-Rothman}, \emph{Time-{{Periodic
  Einstein}}--{{Klein}}--{{Gordon Bifurcations}} of {{Kerr}}}, Commun. Math.
  Phys. 356 (2017), no.~3, 1155--1250.

\bibitem[DHR19]{Daf.Hol.Rod19a}
M.~Dafermos, G.~Holzegel, I.~Rodnianski, \emph{Boundedness and decay for the
  {{Teukolsky}} equation on {{Kerr}} spacetimes {{I}}: The case a {$<<$}
  {{M}}}, Ann. PDE 5 (2019), no.~1, 118 pp.

\bibitem[Dol17]{Dol17}
D.~Dold, \emph{Unstable mode solutions to the {{Klein}}--{{Gordon}} equation in
  {{Kerr-anti-de Sitter}} spacetimes}, Comm. Math. Phys. 350 (2017), no.~2,
  639--697.

\bibitem[DRS09]{Dia.Rea.San09}
{\'O}.~J. Dias, H.~S. Reall, J.~E. Santos, \emph{Kerr-{{CFT}} and gravitational
  perturbations}, J. High Energy Phys. 2009 (2009), no.~08, 101--101.

\bibitem[DSW15]{Dia.San.Way15}
{\'O}.~J.~C. Dias, J.~E. Santos, B.~Way, \emph{Black holes with a single
  {{Killing}} vector field: Black resonators}, J. High Energ. Phys. 2015
  (2015), no.~12, 1--10.

\bibitem[Fri95]{Fri95}
H.~Friedrich, \emph{Einstein equations and conformal structure: {{Existence}}
  of anti-de {{Sitter-type}} space-times}, Journal of Geometry and Physics 17
  (1995), no.~2, 125--184.

\bibitem[GH23]{Gra.Hol23}
O.~Graf, G.~Holzegel, \emph{Mode stability results for the {{Teukolsky}}
  equations on {{Kerr}}--anti-de {{Sitter}} spacetimes}, Class. Quantum Grav.
  40 (2023), no.~4, 045003.

\bibitem[GH24a]{Gra.Hol24}
O.~Graf, G.~Holzegel, \emph{Linear {{Stability}} of {{Schwarzschild-Anti-de
  Sitter}} spacetimes {{I}}: {{The}} system of gravitational perturbations},
  arXiv:2408.02251  (2024), 54 pp.

\bibitem[GH24b]{Gra.Hol24a}
O.~Graf, G.~Holzegel, \emph{Linear {{Stability}} of {{Schwarzschild-Anti-de
  Sitter}} spacetimes {{II}}: {{Logarithmic}} decay of solutions to the
  {{Teukolsky}} system}, arXiv:2408.02252  (2024), 48 pp.

\bibitem[GHIW16]{Gre.Hol.Ish.Wal16}
S.~R. Green, S.~Hollands, A.~Ishibashi, R.~M. Wald, \emph{Superradiant
  instabilities of asymptotically anti-de {{Sitter}} black holes}, Class.
  Quantum Grav. 33 (2016), no.~12, 125022.

\bibitem[Hol09]{Hol09}
G.~Holzegel, \emph{On the massive wave equation on slowly rotating {{Kerr-AdS}}
  spacetimes}, Comm. Math. Phys. 294 (2009), no.~1, 169--197.

\bibitem[HR99]{Haw.Rea99}
S.~W. Hawking, H.~S. Reall, \emph{Charged and rotating {{AdS}} black holes and
  their {{CFT}} duals}, Phys. Rev. D 61 (1999), no.~2, 024014.

\bibitem[HS13]{Hol.Smu13}
G.~Holzegel, J.~Smulevici, \emph{Decay properties of {{Klein-Gordon}} fields on
  {{Kerr-adS}} spacetimes}, Comm. Pure Appl. Math. 66 (2013), no.~11,
  1751--1802.

\bibitem[HW14]{Hol.War14}
G.~H. Holzegel, C.~M. Warnick, \emph{Boundedness and growth for the massive
  wave equation on asymptotically anti-de {{Sitter}} black holes}, Journal of
  Functional Analysis 266 (2014), no.~4, 2436--2485.

\bibitem[Kha83]{Kha83}
U.~Khanal, \emph{Rotating black hole in asymptotic de {{Sitter}} space:
  {{Perturbation}} of the space-time with spin fields}, Phys. Rev. D 28 (1983),
  no.~6, 1291--1297.

\bibitem[Mil24]{Mil24}
P.~Millet, \emph{Geometric background for the {{Teukolsky}} equation
  revisited}, Rev. Math. Phys. 36 (2024), no.~03, p. 2430003.

\bibitem[Olv61]{Olv61}
F.~W.~J. Olver, \emph{Error bounds for the {{Liouville-Green}} (or {{WKB}})
  approximation}, Mathematical Proceedings of the Cambridge Philosophical
  Society 57 (1961), no.~4, 22 pp.

\bibitem[{Shl}14]{Shl14}
Y.~{Shlapentokh-Rothman}, \emph{Exponentially {{Growing Finite Energy
  Solutions}} for the {{Klein}}--{{Gordon Equation}} on {{Sub-Extremal Kerr
  Spacetimes}}}, Comm. Math. Phys. 329 (2014), no.~3, 859--891.

\bibitem[Teu72]{Teu72}
S.~A. Teukolsky, \emph{Rotating {{Black Holes}}: {{Separable Wave Equations}}
  for {{Gravitational}} and {{Electromagnetic Perturbations}}}, Phys. Rev.
  Lett. 29 (1972), no.~16, 1114--1118.

\bibitem[War15]{War15}
C.~M. Warnick, \emph{On quasinormal modes of asymptotically anti-de {{Sitter}}
  black holes}, Comm. Math. Phys. 333 (2015), no.~2, 959--1035.

\bibitem[Zhe24a]{Zhe24}
W.~Zheng, \emph{Asymptotically {{Anti-de Sitter Spherically Symmetric Hairy
  Black Holes}}}, arXiv:2410.04758  (2024), 41 pp.

\bibitem[Zhe24b]{Zhe24a}
W.~Zheng, \emph{Exponentially-growing {{Mode Instability}} on
  {{Reissner-Nordstr{\"o}m--Anti-de-Sitter}} black holes}, arXiv:2410.04750
  (2024), 47 pp.

\end{thebibliography}
\end{document}